\documentclass[twoside,11pt,reqno]{amsart}
\usepackage{amsmath,amssymb,amscd,mathrsfs,amscd}
\usepackage{graphics,verbatim, hyperref}
\usepackage{amsthm}
\usepackage{graphicx}
\usepackage{amsfonts}
\usepackage{amsopn}
\usepackage{amsmath, amsfonts, amssymb}
 
\newtheorem{thm}{Theorem}[subsection]

\newtheorem{lem}[thm]{Lemma}
\newtheorem{prop}[thm]{Proposition}

\newcommand{\Z}{\mathbb{Z}}
\newcommand{\N}{\mathbb{N}}
\newcommand{\Q}{\mathbb{Q}}
\newcommand{\U}{\mathbf{U}}

\newcommand{\V}{\mathbf{V}}
\newcommand{\T}{\mathbf{T}}
\newcommand{\s}{\mathbf{S}}
\newcommand{\End}{\operatorname{End}}

\newcommand{\fg}{\mathfrak{g}}
\newcommand{\fh}{\mathfrak{h}}
\newcommand{\0}{\bar 0}
\newcommand{\1}{\bar 1}
\newcommand{\Lmn}{\ensuremath{\Lambda(m|n)}}
\newcommand{\Lmnd}{\ensuremath{\Lambda(m|n,d)}}
\newcommand{\Pmn}{\ensuremath{P(m|n)}}
\newcommand{\ad}{\ensuremath{\operatorname{ad}}}
\newcommand{\Uq}{\ensuremath{\mathbf{U}}}
\newcommand{\Sqmnd}{\ensuremath{S_q(m|n,d)}}
\newcommand{\Vq}{\ensuremath{\mathbf{V}}}
\newcommand{\Sq}{\ensuremath{\mathbf{S}}}
\newcommand{\Tq}{\ensuremath{\mathbf{T}}}
\newcommand{\Aq}{\ensuremath{\mathbf{A}}}
\newcommand{\Iq}{\ensuremath{\mathbf{I}}}
\newcommand{\Yq}{\ensuremath{\mathbf{Y}}}
\newcommand{\Hq}{\ensuremath{\mathbf{H}_{q}}}
\newcommand{\Ures}{\ensuremath{\Uq_{\mathcal{A}}}}
\newcommand{\Ares}{\ensuremath{\Aq_{\mathcal{A}}}}
\newcommand{\Tres}{\ensuremath{\Tq_{\mathcal{A}}}}
\newcommand{\AAA}{\ensuremath{\mathcal{A}}}

\numberwithin{equation}{subsection}

\begin{document}

\title{Presenting Schur Superalgebras}

\author{Houssein El Turkey}
\address{Department of Mathematics \\
          University of Oklahoma \\
          Norman, OK 73019}
\email{houssein@ou.edu}
\author{Jonathan R. Kujawa}
\address{Department of Mathematics \\
          University of Oklahoma \\
          Norman, OK 73019}
\thanks{Research of the second author was partially supported by NSF grant
DMS-0734226 and NSA grant H98230-11-1-0127}\
\email{kujawa@math.ou.edu}
\date{\today}
\begin{abstract} We provide a presentation of the Schur superalgebra and its quantum analogue which generalizes the work of Doty and Giaquinto for Schur algebras.  Our results include a basis for these algebras and a presentation using weight idempotents in the spirit of Lusztig's modified quantum groups.
\end{abstract}

\maketitle

\section{Introduction}

\subsection{The Schur algebra}  The Schur algebra plays a central role in the representation theory of $\operatorname{GL}(n)$ (e.g.\ see \cite{DDPW}).   It is also the prototypical example of a quasihereditary algebra (cf.\  \cite{CPS}).  And, of course, it is at center stage in Schur-Weyl duality.  If $V$ denotes a $n$-dimensional vector space and $V^{\otimes d}$ denotes the $d$-fold tensor product of $V$ with itself (all vector spaces and tensor products are over the rational numbers), then there is action of the symmetric group on $d$ letters, $\Sigma_{d}$, on $V^{\otimes d}$ by permuting the tensor factors.  With this notation we can define the Schur algebra by
\[
S(n,d)=\End_{\Sigma_{d}}\left(V^{\otimes d} \right).
\]  On the other hand the enveloping algebra of the Lie algebra $\mathfrak{gl}(n)$, $U(\mathfrak{gl}(n))$, has a natural action on $V$ and, hence, on $V^{\otimes d}$.  We could instead define $S(n,d)$ as the image of the resulting representation $U(\mathfrak{gl}(n)) \to \End_{\Q}\left(V^{\otimes d} \right)$.  Schur-Weyl duality implies these two definitions coincide.  Thus the Schur algebra acts as a bridge between representations of $\mathfrak{gl}(n)$ and the symmetric group.  The above story generalizes to the quantum setting if we replace the rational numbers with the rational functions in the indeterminate $q$, the symmetric group by its Iwahori-Hecke algebra, and the enveloping algebra by the quantum group associated to $\mathfrak{gl}(n)$.  The resulting algebra is called the $q$-Schur algebra.

Because of fundamental importance of the Schur and $q$-Schur algebras it is desirable to study them from as many perspectives as possible.  Building on work of Green \cite{Green},  Doty and Giaquinto provided a presentation of the Schur algebras by generators and relations \cite{DG}. Since the enveloping algebra surjects onto the Schur algebra, the known generators and relations for $U(\mathfrak{gl}(n))$ yield generators and relations for the Schur algebra.   But as $U(\mathfrak{gl}(n))$ is infinite dimensional and $S(n,d)$ is finite dimensional, there must be additional relations.  Remarkably,  Doty and Giaquinto  prove that only two more, easy to state, relations are required.  As an outcome of their calculations they obtain a basis and a presentation via weight idempotents reminiscent of Lusztig's modified quantum group, $\dot{\Uq}$.  They also prove quantum analogues of all these results.

One notable application of the Doty-Giaquinto presentation is Li's recent geometric realization of Schur algebras as a certain ring of constructible functions on generalized Steinberg varieties \cite{Li}.  We also see that their presentation of the $q$-Schur algebra is closely related to the geometric construction of the $q$-Schur algebras and quantum group $\Uq_{q}(\mathfrak{gl}(n))$ given by Beĭlinson,  Lusztig, and MacPherson \cite{BLM} (cf.\  \cite[Part 5]{DDPW}).

\subsection{The Schur superalgebra}  There is a $\Z_{2}$-graded (i.e.\ ``super'') analogue of the above setup.  Namely, now let $V=V_{\0} \oplus V_{\1}$ denote a $\Z_{2}$-graded vector space with the dimension of $V_{\0}$ equal to $m$ and the dimension of $V_{\1}$ equal to $n$.  We define $V^{\otimes d}$ as the $d$-fold tensor product of $V$ with itself.  The symmetric group $\Sigma_{d}$ acts on $V^{\otimes d}$ by signed permutation of the tensor factors.  The Schur superalgebra is then defined to be
\[
S(m|n,d) = \End_{\Sigma_{d}}\left(V^{\otimes d} \right).
\]  On the other hand the enveloping superalgebra of the Lie superalgebra $\mathfrak{gl}(m|n)$, $U(\mathfrak{gl}(m|n))$, has a natural action on $V$ and, hence, on $V^{\otimes d}$.  We could instead define $S(m|n,d)$ as the image of the resulting representation $U(\mathfrak{gl}(m|n)) \to \End_{\Q}\left(V^{\otimes d} \right)$.  The super version of Schur-Weyl duality implies these two definitions coincide \cite{BR, Sergeev}.  Thus the Schur superalgebra acts as a bridge between representations of $\mathfrak{gl}(m|n)$ and the symmetric group.  In positive characteristic this connection can be used to prove the Mullineux Conjecture \cite{BK}.

There is also a quantum version of this story.   We again replace the rational numbers with the rational functions in the indeterminate $q$ and the symmetric group by its Iwahori-Hecke algebra, and now replace the enveloping algebra by the quantum group associated to $\mathfrak{gl}(m|n)$. Schur-Weyl duality in this setting was established by Moon \cite{Moon} and Mitsuhashi \cite{M}.  The resulting algebra is called the $q$-Schur superalgebra.  Recently Du and Rui studied the representation theory and combinatorics of the $q$-Schur superalgebras  \cite{DR}.

\subsection{Results}  In this paper we generalize the results of Doty-Giaquinto to the Schur and $q$-Schur superalgebras.  It should be noted that after obtaining the appropriate analogues of the ingredients used in \cite{DG}, the final results are proved using the same arguments as in the non-super case.  The main challenge is to correctly formulate and prove these analogues.

In Theorem~\ref{T:t1} we obtain a presentation for the Schur superalgebra from the standard presentation of the enveloping algebra for $\mathfrak{gl}(m|n)$.  We prove we only need to add two additional relations just as in the case of the Schur algebra.  We then give an explicit basis for the Schur superalgebra and its integral form in Theorem~\ref{T:t1b}.  Finally, in Theorem~\ref{T:t1c} we prove that the Schur superalgebra admits a presentation using weight idempotents in a form reminiscent of Lusztig's modified quantum group.

We also prove the analogous results in the quantum setting.  We use the quantum group $\Uq = U_{q}(\mathfrak{gl}(m|n))$ as presented by Zhang \cite{Z} and prove in Theorem~\ref{T:t2} that we need to add only two additional relations to the standard presentation of $\Uq$ to obtain the $q$-Schur superalgebra.  We also provide a basis for the $q$-Schur superalgebra and an $\AAA = \Z[q,q^{-1}]$-form in Theorem~\ref{T:integralquantum}.  Finally, in Theorem~\ref{T:quantumidempotent}  we prove that the $q$-Schur superalgebra admits a presentation via weight idempotents which is reminiscent of Lusztig's modified quantum group for $\mathfrak{gl}(n)$.

\subsection{Future directions} The results of this paper open the door to a number of interesting avenues of research.  Sergeev \cite{Sergeev} and Olshanski \cite{Ol}, in the nonquantum and quantum cases, respectively, give a Schur-Weyl duality for the type Q Lie superalgebras.  It would be interesting to obtain a presentation for the corresponding type Q Schur superalgebras.   In a different direction, our presentation of the Schur and $q$-Schur superalgebras \`a la Doty-Giaquinto suggests the possibility of geometric constructions for $\mathfrak{gl}(m|n)$ in the spirit of \cite{BLM, Li}.  In a third direction, in proving the quantum case we obtain the commutator formulas for the divided powers of root vectors and establish the existence of an $\AAA = \Z[q,q^{-1}]$-form for the quantum group $\Uq_{q}(\mathfrak{gl}(m|n))$.  Although perhaps not surprising to experts, to our knowledge this has not appeared elsewhere in the literature.  The existence of such a form allows one to consider representations at a root of unity and a super analogue of Lusztig's small quantum group as in \cite{Lusztig}.  Finally, the existence of a presentation of the $q$-Schur superalgebra using weight idempotents suggests that Lusztig's modified quantum groups should have a super analogue.  Lusztig's modified quantum group is a key ingredient to the categorification of the quantum group associated to $\mathfrak{sl}(n)$ (for example, as explained in \cite{Lauda}). Also see \cite{MSV} and references therein for a discussion of categorifications of the $q$-Schur algebras.  The categorification of quantum supergroups is currently an open problem and a super analogue of Lusztig's modified quantum group may be useful.

\subsection{Acknowledgments}  The authors would like to thank the anonymous referee for suggesting improvements to the exposition. 

\section{Nonquantum Case}
In this section all vector spaces will be over the rational numbers, $\Q$.
\subsection{The Lie superalgebra $\mathfrak{gl}(m|n)$}\label{SS:glmndef}   Given a $\Z_{2}$-graded vector space $V=V_{\0}\oplus V_{\1}$ we write $\overline{v} \in \Z_{2}$ for the degree of a homogeneous element $v \in V$.  For short we call $v$ \emph{even} (resp.\ \emph{odd}) if $\overline{v}=\0$ (resp.\ $\overline{v}=\1$).  Let us also introduce the following convenient notation.  For fixed nonnegative integers $m$ and $n$ and $1\leq i \leq m+n$ we define
\begin{equation}\label{E:parity}
\overline{i}=\begin{cases}
            \0,& \text{if $i\leq m$};\\
            \1,& \text{if $i\geq m+1$.}
            \end{cases}
\end{equation}

Let $\fg=\fg_{\0}\oplus \fg_{\1}$ denote the Lie superalgebra $\mathfrak{gl}(m|n)$.   As a vector space $\fg$ is the set of $m+n$ by $m+n$ matrices.  For $1\leq i,j \leq m+n$ we set $E_{i,j}$ to be the matrix unit with a $1$ in $i$th row and $j$th column.  Then the set of matrix units forms a homogeneous basis for $\fg$.  The $\Z_{2}$-grading on $\fg$ is defined by setting
$\fg_{\0}$ to be the span of $E_{i,j}$ where $1\leq i,j  \leq m$ or $m+1 \leq i,j \leq m+n$ and $\fg_{\1}$ to be the span of the $E_{i,j}$ such that $m+1\leq i \leq m+n$ and $1\leq j \leq n$ or $1\leq i \leq m$ and $m+1 \leq j \leq m+n$.  That is, the degree of $E_{i,j}$ is $\overline{i}+\overline{j}$.

The Lie bracket on $\fg$ is given by the supercommutator,
\begin{equation}\label{E:bracketdef}
[E_{ij},E_{kl}]=\delta_{jk}E_{il}-(-1)^{\overline{E}_{ij}\overline{E}_{kl}}\delta_{il}E_{kj}.
\end{equation}  By definition it is bilinear and so it suffices to define it on the basis of matrix units.

 We fix $\fh$ to be the Cartan subalgebra of $\fg$ consisting of all diagonal matrices and let $\fh^*$ be its dual. Let $\varepsilon_i:\fh\rightarrow \Q$ be the linear map that takes an element of $\fh$ to its $i$th diagonal entry. The set $\{\varepsilon_i \mid  1\leq i\leq m+n\}$ forms a basis of $\fh^*$ and we define a bilinear form, $(\,,\,)$, on $\fh^*$ by setting
\begin{equation}\label{E:bilinearform}
(\varepsilon_i,\varepsilon_j)=(-1)^{\overline{i}}\delta_{ij}.
\end{equation}
With our choice of Cartan subalgebra the root system of $\fg$ is
\[
\Phi=\{\varepsilon_i-\varepsilon_j\mid 1\leq i\neq j\leq m+n\}
\] and the matrix unit $E_{i,j}$ spans the $\varepsilon_{i}-\varepsilon_{j}$ root space.  In particular there is a natural $\Z_{2}$-grading on $\Phi$ given by declaring that the root $\varepsilon_{i}-\varepsilon_{j}$ has degree $\overline{E}_{i,j} = \overline{i}+\overline{j}$. We fix the Borel subalgebra of $\fg$ given by taking all upper triangular matrices.  Corresponding to this choice of Borel the positive roots are
\[
\Phi^+=\{\varepsilon_i-\varepsilon_j\mid 1\leq i<j \leq m+n\}
\] and if we set $\alpha_i=\varepsilon_i-\varepsilon_{i+1}$, then $\{\alpha_1,\ldots,\alpha_{m+n-1}\}$ are the simple roots.  The simple roots have degree
  \[\overline{\alpha}_{i}=\begin{cases}
            \0,& \text{if $i\neq m$};\\
            \1,& \text{if $i=m$.}
            \end{cases}\]

\subsection{The Schur superalgebra}\label{SS:schursuperalgebra}

A $\fg$-(super)module is a $\Z_{2}$-graded vector space $M=M_{\0}\oplus M_{\1}$ which admits an action by $\fg$.  The action respects the $\Z_{2}$-grading in that for any $r,s \in \Z_{2}$, if $x \in \fg_{r}$ and $m \in M_{s}$, then $x.m \in M_{r+s}$.  The action also respects the Lie bracket in that for any homogeneous $x,y \in \fg$ and $m\in M$, we have
\[
[x,y].m= x.(y.m) - (-1)^{\overline{x} \cdot \overline{y}}y.(x.m).
\]
  Note that here and elsewhere we give the condition only on homogeneous elements. The general case is obtained by linearity.  As all modules will be $\Z_{2}$-graded, we choose to omit the prefix ``super''.

The natural $\fg$-module, $V$, is the vector space of column vectors of height $m+n$.  For $1 \leq i \leq m+n$, let $v_{i}$ denote the element of $V$ with a $1$ in the $i$th row and zeros elsewhere.  Then the set $\{v_{i} \mid 1 \leq i\leq m+n \}$ defines a homogeneous basis for $V$ with $\overline{v}_{i}=\overline{i}$ for $i=1, \dotsc , m+n$.  The action of $\fg$ on $V$ is given by left multiplication.

We denote universal enveloping superalgebra of $\fg$ by $U$. It inherits a $\Z_{2}$-grading from $\fg$ and natural basis given by the PBW theorem for Lie superalgebras \cite[Section 1.1.3]{K}.    As for Lie algebras, a $\fg$-module can naturally be thought of as a $U$-module and vice versa.  In particular, $U$ admits a coproduct and so if $M$ and $N$ are $\fg$-modules, then $M \otimes N$ is again a $\fg$-module.

As it will be important in the calculations which follow, let us make this explicit.  The coproduct $U \to U \otimes U$ is given on elements of $\fg$ by $x \mapsto x \otimes 1 + 1 \otimes x$.  We use the convention that in any formula in which two homogenous elements have their order reversed, a sign is introduced which is $-1$ raised to the product of their degrees.  Given a homogeneous element $x\in \fg$ and homogeneous $m\in M$ and $n \in N$, then the coproduct along with the sign convention implies that we have
\[
x.(m\otimes n) = (x.m)\otimes n + (-1)^{\overline{x}\cdot \overline{ m}}m \otimes(x.n).
\]

In particular, for $d \geq 1$ we may define the $d$-fold tensor product of the natural module,
\[
V^{\otimes d} := V \otimes V \otimes \dotsb \otimes V.
\]
 Let
\[
\rho_d:U \to \End_{\Q} \left( V^{\otimes d }\right)
\] denotes the corresponding superalgebra homomorphism.  We define the \emph{Schur superalgebra} $S(m|n,d)$ to be the image of $\rho_d$.  In particular, we can and will think of $S(m|n,d)$ as a quotient of $U$.

Note that the Schur superalgebra can also be defined as follows.  There is a signed permutation action of the symmetric group on $d$ letters, $\Sigma_{d}$, on $V^{\otimes d}$.  The super analogue of Schur-Weyl duality \cite{BR, Sergeev} then shows that
\[
S(m|n,d) = \End_{\Sigma_{d}}\left(V^{\otimes d} \right).
\]
\subsection{A presentation of the Schur superalgebra}

 Our first main result gives the Schur superalgebra  by generators and relations.  Here and throughout, if $A$ is an associative superalgebra and  $x,y \in A$ are homogeneous elements, then we write
\[
[x,y] = xy - (-1)^{\overline{x}\cdot \overline{y}}yx.
\]  For an element $x \in A$ the map $\ad x: A \to A$ is defined by $\ad x (y) = [x,y]$.  Note that the bilinear form used in the following relations is the one introduced in \eqref{E:bilinearform}.
 \begin{thm}\label{T:t1} The Schur superalgebra $S(m|n,d)$ is generated by homogeneous elements
\[
e_{1}, \dotsc , e_{m+n-1}, f_{1}, \dotsc , f_{m+n-1}, H_{1}, \dotsc , H_{m+n}
\]
where the $\Z_{2}$-grading is given by setting  $\overline{e}_{m}=\overline{f}_{m}=\1$,  $\overline{e}_{i}= \overline{f}_{i}=\0$ for $i \neq m$, and $\overline{H}_{i}=\0$.

  The following is a complete set of relations:
 \begin{enumerate}
 \item [(R1)]$[H_i,H_j]=0, \text{ where }1\leq i,\,j\leq m+n$;\\
\item [(R2)]$[e_i,f_j]=\delta_{ij}\left(H_i-(-1)^{\overline{e}_i \cdot \overline{f}_j}H_{j+1} \right), \quad 1\leq i,\,j\leq m+n-1$;\\
\item [(R3)] $[H_i,e_j]=(-1)^{\overline{i}}(\varepsilon_i,\alpha_j)e_j,\quad\text{and}\quad [H_i,f_j]=-(-1)^{\overline{i}}(\varepsilon_i,\alpha_j)f_j,\quad \\  \text{where }1\leq i\leq m+n,\quad 1\leq j\leq m+n-1$;\\
\item [(R4)] $[e_m,e_m]=0,\quad (\ad e_i)^{1+|(\alpha_i,\alpha_j)|}e_j=0,\quad\text{if $1\leq i\neq j\leq m+n-1$ and $i\neq m$}$\\$[e_m,[e_{m-1},[e_m,e_{m+1}]]]=0,\quad\text{if $m, n\geq2$}$; \\
 \item [(R5)]$[f_m,f_m]=0,\quad (\ad f_i)^{1+|(\alpha_i,\alpha_j)|}f_j=0,\quad\text{if $1\leq i\neq j\leq m+n-1$ and $i\neq m$}$\\
 $[f_m,[f_{m-1},[f_m,f_{m+1}]]]=0\quad\text{if $m,n \geq2$}$; \\
 \item [(R6)]$H_1+H_2+\cdots+H_{m+n}=d$; \\
\item [(R7)] $H_i(H_i-1)\cdots (H_i-d)=0,$ where $1\leq i\leq m+n$.
 \end{enumerate}
 \end{thm}

\subsection{Strategy and simplifications}\label{SS:notations} 

The basic strategy of the proof of Theorem~\ref{T:t1} is as in \cite{DG} and as follows.  For short, let us write $S$ for $S(m|n,d)$.   Let $T$ be the superalgebra given by the generators and relations in the theorem.  The goal is to prove $T$ is isomorphic to $S$ as superalgebras.  We first show that relations  $(R1)$-$(R7)$ hold in $S$.  This implies we have a surjective homomorphism $T \to S$.  We then prove that the dimension of $T$ is no larger than the dimension of $S$ by exhibiting a spanning set of $T$ with cardinality equal to the dimension of $S$.  See Section~\ref{SS:integralformbasis}.  This immediately implies that the map is an isomorphism and the spanning set is a basis.

Note that the universal enveloping superalgebra $U$ is the superalgebra on the same generators but subject only to the relations $(R1)$-$(R5)$ (see \cite{LS} or \cite{Z2}). As $S(m|n,d)$ is a quotient of $U$ via $\rho_{d}$ it has the same generators but possibly additional relations. The content of Theorem~\ref{T:t1} is that we only need to add relations $(R6)$ and $(R7)$ to obtain a presentation of $S(m|n,d)$.

As it will be helpful in later calculations, let us briefly pause to make explicit the connection between this presentation of $U$ via generators and relations and the one obtained from the matrix realization of $\fg$ given in Section~\ref{SS:glmndef}.   If we write $E_{i,j}$ for the $ij$-matrix unit as in Section~\ref{SS:glmndef}, then the isomorphism between these superalgebras is given on generators by  $e_{i}\mapsto E_{i,i+1}$, $f_{i}\mapsto E_{i+1,i}$, and $H_{i}\mapsto E_{i,i}$.  We identify these two realizations of $U$ via this map.    In particular, there is a canonical embedding $\fg  \hookrightarrow U$ and we will identify $\fg$ with its image under this map.

As both $S$ and $T$ are quotients of $U$ they are both generated by the images of generators of $U$.  To lighten notation, we choose to use the same notation for algebra elements which can be viewed in more than one of these algebras.  In particular, we write $e_{i}$, $f_{i}$, and $H_{i}$ for the generators of $U$ and their images in $S$ and $T$.   We will endeavor to always be clear in which algebra we are working.  If the algebra is not explicitly stated, then the calculation holds for all three algebras $U$, $S$, and $T$.

We will also frequently make use of the fact that the inclusion
\[
\mathfrak{gl}(m)\oplus\mathfrak{gl}(n) \cong \fg_{\0} \subseteq \mathfrak{gl}(m|n)
\]
induces an inclusion
\[
U(\mathfrak{gl}(m)\oplus\mathfrak{gl}(n)) \hookrightarrow  U(\mathfrak{gl}(m|n)).
\] Thus any computation involving purely even elements will carry over from \cite{DG}.  More generally, when calculations are essentially identical to those in \cite{DG} we will usually leave them to the reader.

\subsection{The new relations}\label{SS:R6R7}  We now observe that the equations $(R6)$ and $(R7)$ hold in $S$.

\begin{lem}
Under the representation $\rho_d:U\rightarrow \End(V^{\otimes d})$ the elements $H_1, \dotsc , H_{m+n}$ in $S$ satisfy the relations $(R6)$ and $(R7)$. Moreover, the relation $(R7)$ is the minimal polynomial of $H_i$ in $ \End_{\Q }\left( V^{\otimes d}\right)$.
\end{lem}

\begin{proof}
 Since the elements $H_1, \dotsc, H_{m+n}$ are purely even, this follows from \cite[Lemma~4.1]{DG}.
\end{proof} 

As explained above, this implies the surjection $\rho_{d}:U \to S$ factors through $T$ and we obtain a surjective superalgebra homomorphism, $T \to S$. To prove that this map is an isomorphism it suffices to show that their dimensions are equal by obtaining an explicit basis for $T$ and, hence, for $S(m|n,d)$.  In fact it turns out to be no harder to work over the integers and so we obtain a basis for an integral form, $S(m|n,d)_{\Z}$, of the Schur superalgebra.

\subsection{Divided powers}\label{SS:dividedpowers} Let $A$ denote any of $U$, $S$, or $T$.  
 Recall from Section~\ref{SS:notations} that we identify  $\mathfrak{gl}(m|n)$ as a subspace of $U$.  For each $\alpha= \varepsilon_{i}-\varepsilon_{j} \in \Phi^{+}$ we use this identification and write $x_{\alpha}$ for the image in $A$ of the matrix unit $E_{i,j}$. We call $x_{\alpha}$ a \emph{root vector}.  For $x \in A$ and $k\in \Z_{\geq 0}$, define the \emph{$k$th divided power of $x$} to be
\[
x^{(k)}=\frac{x^k}{k!}.
\]  In particular, we have the divided powers of the root vectors, $x_{\alpha}^{(r)}$, for all $\alpha \in \Phi$ and $r \geq 0$.

 We define
\begin{equation*}
\Lmn=\left\{ \lambda= (\lambda_{1}, \dotsc , \lambda_{m+n})\mid \lambda_{i} \in \Z , \lambda_{i}\geq 0 \text{ for } 1 \leq i\leq m+n \right\}.
\end{equation*}  Given any tuple of integers $\lambda$ (e.g.\ $\lambda \in \Lmn$), let $|\lambda|$ denote the sum of those integers.  Using this we define
\[
\Lmnd=\left\{\lambda\in \Lmn\mid |\lambda|=d \right\}.
\]  For $i=1, \dotsc , m+n$ and $k\geq 0$ define an element of $A$ by
\[
\binom {H_{i}}{k}=\frac{H_i(H_i-1)\cdots (H_i-k+1)}{k!},
\] where, by definition,
\[
\binom{H_i}{0}=1.
\]

\subsection{The Kostant $\Z$-form}\label{SS:integralform}  

 We now define analogues of the Kostant $\Z$-form. We also take this opportunity to introduce certain subalgebras which will be needed in what follows.   Let $A$ denote $U$, $S$, or $T$.  Let $A^0$ denote the subsuperalgebra of $A$ generated by $H_1, \dotsc , H_{m+n}$. In particular, if $A$ is $S$ or $T$, then it is clear that $A^{0}$ is the image of $U^{0}$ respectively, under the quotient map.

The \emph{Kostant $\Z$-form} for $A$ is denoted by $A_{\Z}$ and it is defined to be the subring of $A$ generated by
\[
\left\{e_{i}^{(k)}, f_{i}^{(k)} \mid i=1, \dotsc , m+n-1, k\geq 0 \right\} \bigcup \left\{\binom{H_{i}}{k} \mid i=1, \dotsc , m+n, k \geq 0 \right\}.
\]  Moreover, we set $A_{\Z}^0$ to be the intersection of $A^0$ with $A_{\Z}$. For $A$ equal to $S$ or $T$, it is clear that $A_{\Z}$ and $A_{\Z}^0$  is nothing but the image of $U_{\Z}$ and $U_{\Z}^0$, respectively, under the quotient map.

\subsection{The weight idempotents} We begin by investigating the structure of $T^0$ and $T^{0}_{\Z}$. For $\lambda=(\lambda_{i})\in \Lmn$ we define
\[
H_\lambda=\prod_{i=1}^{m+n} \binom{H_i}{\lambda_i}.
\]  Note that as $H_{1}, \dotsc , H_{m+n}$ commute the the product can be taken in any order.  When $\lambda \in \Lmnd$ it is convenient to set the notation
\[
1_\lambda=H_\lambda.
\]  Because of part (b) of the following proposition we refer to these elements as \emph{weight idempotents}.
\begin{prop}\label{p1}
Let $I^0$ be the ideal of $U^{0}$ generated by the elements
\[
H_1+H_2+\cdots+H_{m+n}-d
\]
and
\[
H_i(H_i-1)\cdots (H_i-d)
\]
for $i=1, \dotsc , m+n$. Then
\begin{itemize}
\item[(a)] We have a superalgebra isomorphism $U^{0}/I^{0}\cong T^0$.
\item[(b)] The set $\{1_\lambda\mid \lambda\in \Lmnd\}$ is a $\Q$-basis for $T^0$ and a $\Z$-basis for $T_{\Z}^0$. Moreover, they give a set of pairwise orthogonal idempotents which sum to the identity.
\item[(c)] In $T^{0}$ we have $H_\lambda=0$ for any $\lambda \in \Lmn$ such that $|\lambda|>d$.
\end{itemize}
\end{prop}

\begin{proof}
 Since the elements $H_{1}, \dotsc , H_{m+n}$ are purely even, this follows from \cite[Proposition~4.2]{DG}.
\end{proof}

\begin{prop}\label{p2}
Let $1\leq i\leq m+n$, $k \in \Z_{\geq 0}$, $\lambda\in \Lmnd$, and $\mu \in \Lmn$. We have the following identities in the superalgebra $T^0$:
\begin{enumerate}
\item $H_i1_\lambda=\lambda_i1_\lambda, \quad \displaystyle{\binom{H_i}{k}1_\lambda=\binom{\lambda_i}{k}1_\lambda}$;
\item $H_\mu 1_\lambda=\lambda_\mu 1_\lambda$, where $\lambda_\mu=\displaystyle{\prod _i \binom{\lambda_i}{\mu_i}}$;
\item $H_\mu=\displaystyle{\sum_{\lambda \in  \Lmnd} \lambda_\mu 1_\lambda}$.
\end{enumerate}
\end{prop}

\begin{proof}  They follow from  \cite[Proposition~4.3]{DG}.
\end{proof}

\subsection{The root vectors}\label{SS:rootvectors} We continue to let $A$ denote any of $U,S$, or $T$.  Recall from Section~\ref{SS:dividedpowers} that for each $\alpha \in \Phi$ we have the root vector $ x_{\alpha}\in A$.  In particular, note that $x_{\alpha}$ is homogeneous and $\overline{x}_{\alpha}=\overline{\alpha}$, where the grading on roots is as given in Section~\ref{SS:glmndef}.
Given $\alpha = \varepsilon_{i}-\varepsilon_{j} \in \Phi$, we set
\[
H_\alpha=H_i-(-1)^{\overline{x}_\alpha}H_j.
\]  Given $\alpha=\varepsilon_{i}-\varepsilon_{j}, \beta = \varepsilon_{k}-\varepsilon_{l} \in \Phi$ such that $\alpha + \beta \in \Phi$, we define
\begin{equation}\label{E:cdef}
c_{\alpha,\beta}=\begin{cases}
             1,& \text{if $j=k$};\\
            -(-1)^{\overline{x}_\alpha\overline{x}_\beta},& \text{if $i=l$}.
            \end{cases}
\end{equation}
Using this notation \eqref{E:bracketdef} implies the following commutator formula for root vectors in $A$.
\begin{lem}\label{l2}
Let $\alpha,\,\beta\in \Phi$ and say $\alpha=\varepsilon_i-\varepsilon_j$ and $\beta=\varepsilon_k-\varepsilon_l$, we have
 \[ [x_\alpha,x_\beta]=\begin{cases}
             H_\alpha,& \text{if $\alpha+\beta=0$};\\
            c_{\alpha,\beta}x_{\alpha+\beta},& \text{if $\alpha+\beta \in \Phi$};\\
            0,& \text{otherwise}.
            \end{cases}\]

\end{lem}

We also note that an easy induction proves that for all $a,b \geq 0$ and $\alpha \in \Phi$ we have
\begin{equation}
 x_\alpha^{(a)}x_\alpha^{(b)}=\binom{a+b}{a}x_\alpha^{(a+b)}.\label{eq9}
\end{equation}

\subsection{Commutation relations between root vectors and weight idempotents} We now compute the commutation relations between root vectors and weight idempotents.
\begin{prop}\label{p3}
For any $\alpha \in \Phi$, $\lambda \in \Lmnd$ we have the commutation formulas:
 \[x_\alpha1_\lambda=\begin{cases}
            1_{\lambda+\alpha}x_\alpha& \text{if $\lambda+\alpha \in \Lmnd$}\\
            0& \text{otherwise}
            \end{cases}\]
and
 \[1_\lambda x_\alpha=\begin{cases}
            x_\alpha1_{\lambda-\alpha}& \text{if $\lambda-\alpha \in \Lmnd$}\\
            0& \text{otherwise.}
            \end{cases}\]
\end{prop}

\begin{proof}
Although analogous to \cite[Proposition 4.5]{DG}, the proof involves keeping track of signs so we include it.  We first note that \eqref{E:bracketdef} implies for all $l = 1, \dotsc , m+n$ and $\alpha \in \Phi$ we can use the parity function given in \eqref{E:parity} and the bilinear form given in \eqref{E:bilinearform} to write
\begin{align}\label{eq2}
[H_l,x_\alpha]=(-1)^{\overline{l}}(\varepsilon_l,\alpha)x_\alpha.
\end{align}
Now say $\alpha = \varepsilon_{i}-\varepsilon_{j}$.  Using \eqref{eq2} we obtain
\begin{align*}
x_\alpha 1_\lambda=&\left( \prod_{l\neq i,j} \binom{H_l}{\lambda_l}\right)\binom{H_i-(-1)^{\overline{i}}(\varepsilon_i,\alpha)}{\lambda_i}\binom
                 {H_j-(-1)^{\overline{j}}(\varepsilon_j,\alpha)}{\lambda_j}x_\alpha\\
                 =&\left(\prod_{l\neq i,j} \binom{H_l}{\lambda_l}\right)\binom{H_i-(-1)^{\overline{i}}(-1)^{\overline{i}}}{\lambda_i}\binom
                 {H_j-(-1)^{\overline{j}}(-(-1)^{\overline{j}})}{\lambda_j}x_\alpha\\
                 =&\left(\prod_{l\neq i,j} \binom{H_l}{\lambda_l}\right)\binom{H_i-1}{\lambda_i}\binom {H_j+1}{\lambda_j}x_\alpha.
\end{align*}
Multiplying on the left by $\displaystyle{\frac{H_i}{\lambda_i+1}}$ and using the fact that $H_ix_\alpha=x_\alpha (H_i+1)$, we get
\[
x_\alpha \frac{H_i+1}{\lambda_i+1} 1_\lambda=\frac{H_i}{\lambda_i+1}\binom{H_i-1}{\lambda_i}\binom
{H_j+1}{\lambda_j}\left( \prod_{l\neq i,j} \binom{H_l}{\lambda_l}\right)x_\alpha
\]
which, using Proposition~\ref{p2}, simplifies to
\begin{equation}\label{E:xalpha1lambda}
x_\alpha1_\lambda=\left(\binom{H_i}{\lambda_i+1}\binom {H_j+1}{\lambda_j}\prod_{l\neq i,j} \binom{H_l}{\lambda_l}\right)x_{\alpha}.
\end{equation}
If $\lambda_j>0$, then this can be rewritten as
\[x_\alpha1_\lambda=\displaystyle{\binom{H_i}{\lambda_i+1}\left(\binom {H_j}{\lambda_j}+\binom {H_j}{\lambda_j-1}\right)\prod_{l\neq i,j}
\binom{H_l}{\lambda_l}x_\alpha.}\]
The first summand on the right-hand-side of the preceding equality vanishes by Proposition~\ref{p1}. This proves the first part of the proposition in the case
$\lambda_j>0$. If $\lambda_j=0$, then \eqref{E:xalpha1lambda} can be written as
 \begin{align*}
 x_\alpha1_\lambda =& \left(\binom{H_i}{\lambda_i+1}\prod_{l\neq i,j} \binom{H_l}{\lambda_l}\right)x_{\alpha}\\
                  = & H_\mu x_\alpha,
\end{align*}
where $\mu=(\lambda_1,\ldots,\lambda_i+1,\ldots,\lambda_{j-1},0,\ldots,\lambda_{m+n})$. But then $|\mu|=|\lambda|+1>d$ and hence the right-hand-side is zero by Proposition~\ref{p1}.  This proves the first statement. The proof of the second is similar.
\end{proof}

\subsection{Commutation relations between divided powers of root vectors}  We now compute the commutation formulas between divided powers of root vectors; but first we make a simplifying observation.  If the root vector $x_\alpha$ is odd (i.e.\ if $\alpha$ is an odd root), then in $\fg$ we have $[x_{\alpha}, x_{\alpha}]=0$.  But in $U$ and, hence, in $S$ and $T$, we have $[x_\alpha,x_\alpha]=2x_\alpha^2$.  Taken together, this implies
\[
x_{\alpha}^{2}=0
\]
in $U$, $S$, and $T$ for all odd $\alpha \in \Phi$.    That is, for odd roots we only need to consider root vectors of divided power one.
\begin{lem}\label{l2'}
Let $\alpha,\,\beta\in \Phi$ and  $r,\,s\in \Z_{\geq 0}$.
\begin{enumerate}
\item  If ${\overline{x}}_\alpha=0$ and ${\overline{x}}_\beta=0$, then
\begin{equation}
x_\alpha^{(r)} x_{\beta}^{(s)}=\begin{cases}
             x_{\beta}^{(s)}x_\alpha^{(r)}+\displaystyle{\sum_{j=1}^{\min(r,s)}x_{\beta}^{(s-j)}\binom {H_\alpha-r-s+2j}{j}
 x_\alpha^{(r-j)}},& \text{if $\alpha+\beta=0$};\\
           x_{\beta}^{(s)} x_\alpha^{(r)}+\displaystyle{\sum_{j=1}^{\min(r,s)}c_{\alpha,\beta}^jx_{\beta}^{(s-j)}x_{\alpha+\beta}^{(j)}
 x_\alpha^{(r-j)}},& \text{if $\alpha+\beta \in \Phi$};\\
            x_{\beta}^{(s)} x_\alpha^{(r)}; & \text{otherwise}.
            \end{cases}\label{eq3}
\end{equation}

\item   ${\overline{x}}_\alpha=0$ and ${\overline{x}}_\beta=1$, then
\begin{equation}
x_\alpha^{(r)} x_{\beta}^{(1)}=\begin{cases}
           x_{\beta}^{(1)} x_\alpha^{(r)}+c_{\alpha,\beta}x_{\alpha+\beta}x_\alpha^{(r-1)},& \text{if $\alpha+\beta \in \Phi$};\\
            x_{\beta}^{(1)} x_\alpha^{(r)}, & \text{if $\alpha+\beta \notin \Phi$}.
            \end{cases}\label{eq4}
\end{equation}

\item If ${\overline{x}}_\alpha=1$ and ${\overline{x}}_\beta=0$, then
\begin{equation}
x_\alpha^{(1)} x_{\beta}^{(r)}=\begin{cases}
           x_{\beta}^{(r)} x_\alpha^{(1)}+c_{\alpha,\beta}x_{\alpha+\beta}x_\beta^{(r-1)},& \text{if $\alpha+\beta \in \Phi$};\\
            x_{\beta}^{(r)} x_\alpha^{(1)}, & \text{if $\alpha+\beta \notin \Phi$}.
            \end{cases}\label{eq45}
\end{equation}

\item If ${\overline{x}}_\alpha=1$ and ${\overline{x}}_\beta=1$, then
\begin{equation}
x_\alpha^{(1)} x_{\beta}^{(1)}=\begin{cases}
             -x_{\beta}^{(1)}x_\alpha^{(1)}+H_{\alpha},& \text{if $\alpha+\beta=0$};\\
           -x_{\beta}^{(1)} x_\alpha^{(1)}+x_{\alpha+\beta}, & \text{if $\alpha+\beta \in \Phi$};\\
            -x_{\beta}^{(1)} x_\alpha^{(1)}, & \text{otherwise}.
            \end{cases}\label{eq5}
\end{equation}
\end{enumerate}
\end{lem}

\begin{proof} As \eqref{eq3} involves purely even root vectors, it follows from the classical case (see \cite[Equations (5.11a)-(5.11c)]{DG}). Equations~\eqref{eq4}~and~\eqref{eq45} are verified by a straightforward induction on $r$. Equation~\eqref{eq5} follows directly from Lemma~\ref{l2}.
\end{proof}

\subsection{Kostant monomials and content functions}\label{SS:contentfunctions} Any product in $A$ of nonzero elements of the form:
 \begin{equation}\label{E:kostantmonomials}
x_\alpha^{(r)},\quad \displaystyle{\binom {H_i}{s}},
\end{equation}
 taken in any order and for any $r, s\in \Z_{\geq 0},$ $ \alpha\in\Phi,$ $ 1\leq i\leq m+n$, will be called a \emph{Kostant monomial}. Note that by \cite[Lemma 2.1]{kujawa} the set of Kostant monomials span $U_{\Z}$ and, hence, $T_{\Z}$ and $S_{\Z}$.  The goal is to find a subset of Kostant monomials which will provide a basis for $T_{\Z}$. 

We now introduce the content function on Kostant monomials.  They will be used as a bookkeeping device in the proof of Proposition~\ref{p4}.  It is defined just as in the classical case \cite[Section 2]{DG}.

The \emph{content function}
\begin{equation}\label{E:contentdef}
\chi: \left\{\text{Kostant monomials} \right\} \rightarrow \bigoplus_{i=1}^{m+n}\Z\varepsilon_i
\end{equation}
is defined as follows.   We first define it on the elements in \eqref{E:kostantmonomials}.  If $\alpha=\varepsilon_i-\varepsilon_j\in\Phi$ and $r \geq 1$, then

 \[
\chi\left( x_\alpha^{(r)}\right)=r\varepsilon_{\operatorname{max}(i,j)}.
\]  If $i=1, \dotsc , m+n$ and  $r \geq 1$, then
\[
\chi\left( \binom {H_i}{r}\right)=0.
\]
 We then extend this definition by declaring $\chi(XY)=\chi(X)+\chi(Y)$ whenever $X,\,Y$ are Kostant monomials.

We also define a \emph{left content function}, $\chi_L$, and \emph{right content function}, $\chi_R$, on the elements given in \eqref{E:kostantmonomials}  by
\begin{align*}
\chi_L(x_\alpha^{(r)})&=r\varepsilon_i, \\
\chi_R(x_\alpha^{(r)})&=r\varepsilon_j, \\
\chi_L\left( \binom {H_i}{s}\right) &=\chi_R\left( \binom {H_i}{s}\right)=0.
\end{align*}
They are defined on general Kostant monomials using the rules $\chi_L(XY)=\chi_L(X)+\chi_L(Y)$ and  $\chi_R(XY)=\chi_R(X)+\chi_R(Y)$  for any Kostant monomials $X$ and $Y$.

In what follows we view elements in the image of the content functions as elements of $\Lmn$ via the map
\begin{equation}\label{E:contenttoweights}
\sum_{i=1}^{m+n}a_i \varepsilon_{i} \mapsto (a_{1}, \dotsc , a_{m+n}).
\end{equation}

\subsection{A lemma on content functions}\label{SS:eafadef}

To label the elements of our basis for the Schur superalgebra, we need to define the following set of tuples of nonnegative integers indexed by the positive roots of $\fg$:
\begin{equation}\label{E:Pmndef}
\Pmn=\left\{ A=(A(\alpha))_{\alpha\in\Phi^+}\mid A(\alpha) \in \Z_{\geq0} \text{ if $\overline{\alpha}=\0$ and } A(\alpha) \in \{0,1 \} \text{ if $\overline{\alpha}=\1$}\right\}.
\end{equation}
   Fix an order on $\Phi^{+}$.  For $A=(A(\alpha)) \in \Pmn$ we define
\begin{align*} e_A&=\prod_{\alpha \in \Phi^+}x_\alpha^{(A(\alpha))},\\
              f_A&=\prod_{\alpha \in \Phi^+}x_{-\alpha}^{(A(\alpha))},
\end{align*}
 where the products defining $e_A$ and $f_A$ are taken according to the fixed order on $\Phi^+$. 

The last ingredient we need is the following partial order on $\Lmn$.  It is defined by declaring for $\lambda=(\lambda_i)$, $\mu=(\mu_i)$ in $\Lmn$ that
\begin{equation}\label{E:partialorderdef}
\lambda \preceq \mu
\end{equation}
if and only if $\lambda_i\leq \mu_i$ for $i=1, \dotsc , m+n$.
 \begin{lem}\label{l3} For $A=(A(\alpha)),\,\,C=(C(\alpha))\in \Pmn,\,\lambda\in\Lmn$ we have
\[\chi(e_A1_\lambda f_C)\preceq \lambda\quad\text{if and only if}\quad\chi_L(1_{\lambda'}e_Af_C)\preceq \lambda'\quad\text{if and only if}\quad\chi_{R}(e_Af_C1_{\lambda''})\preceq \lambda'',\] where
\begin{align*}
\lambda':=\lambda+\sum_{\alpha \in \Phi^+}A(\alpha)\alpha, \\
\lambda'':=\lambda+\sum_{\alpha\in \Phi^+}C(\alpha)\alpha.
\end{align*}
\end{lem}

 \begin{proof}  As our content functions are defined just as in \cite{DG}, the proof of \cite[Lemma~5.1]{DG} applies verbatim.
  \end{proof}

\subsection{A basis for the Schur superalgebra}\label{SS:integralformbasis}

Let us define the set
 \[
Y=\bigcup_{\substack{\lambda\in\Lmnd\\ A,C \in \Pmn}}\{e_A1_\lambda f_C\mid \chi(e_Af_C)\preceq \lambda \}.
\]
Note that we have the following alternate descriptions of $Y$.  Following from Proposition \ref{p3} we have
\begin{equation*}
e_A1_\lambda f_C=1_{\lambda'}e_Af_C=e_Af_C1_{\lambda''},
\end{equation*} where $\lambda'$ and $\lambda''$ are as above.  Using this and Lemma~\ref{l3} we can characterize $Y$ as
\[
Y=\bigcup_{\substack{\lambda'\in\Lmnd\\ A,C \in \Pmn}}\{1_{\lambda'}e_Af_C \mid \chi_L(e_Af_C)\preceq \lambda'\}=\bigcup_{\substack{\lambda''\in\Lmnd\\ A,C \in \Pmn}}\{e_Af_C1_{\lambda''}\mid \chi_R(e_Af_C)\preceq \lambda''\}.
\]

Finally we are prepared to give a basis for $T$.

 \begin{prop}\label{p4}
 The set $Y$ spans the $\Z$-superalgebra $T_{\Z}$.
 \end{prop}

\begin{proof}  The proof is exactly parallel to the proof of \cite[Proposition~5.2]{DG}.  Namely, as discussed in Section~\ref{SS:contentfunctions}, the Kostant monomials span $T_{\Z}$.  From Proposition~\ref{p1} we in fact know that $T_{\Z}$ is spanned by Kostant monomials consisting of products of divided powers of root vectors and weight idempotents.    Given such a Kostant monomial, we may use Proposition~\ref{p3} to move all weight idempotents to the right hand side of the Kostant monomial.  Thus it suffices to show that Kostant monomials consisting of products of divided powers of root vectors can be written as an integral linear combination of elements in $Y$.  This is done by inducting on the degree and content of the monomial using the commutation formulas.  As our content formula and commutation formulas are of the same form as in \cite{DG}, the inductive argument used there applies here without change.  The only difference appears when we use the commutation formulas given in Lemma~\ref{l2'}. Extra signs appear but all coefficients remain integral and this is all that is needed for the proof.

We also need that for all $s, t \in \Z_{\geq 0}$, the term $\displaystyle{\binom{H_\alpha-t}{s}}$ in \eqref{eq3} belongs to $T_{\Z}^0$. As these elements are purely even this follows from the remark after \cite[Equation (5.11)]{DG}.  It can also be verified directly by an inductive argument using the identity
  \[\binom{H_\alpha-1}{s}=\binom{H_\alpha}{s}-\binom{H_\alpha-1}{s-1}.\]
 \end{proof}

 \begin{lem}\label{l4}
 The cardinality of the set $Y$ is equal to the dimension of the Schur superalgebra.
 \end{lem}

 \begin{proof}
By \cite[Section 2.3]{D} the dimension of the Schur superalgebra is equal to the number of monomials of total degree $d$ in the free supercommutative superalgebra in $m^2+n^2$ even variables and $2mn$ odd variables.
Equivalently, the dimension of $S$ is the same as the number of monomials in $m^2+n^2-1$ even variables and $2mn$ odd variables of total degree not exceeding $d$. From this it is immediate that the
dimension of $S$ is the same as the cardinality of the set
\[
P=\left\{ e_A H_B f_C\mid  B = (B_{i})\in \Lmn;\,B_1=0;\,A,\,C\in \Pmn, |A|+|B|+|C|\leq
d\right\}.
\]
Thus to prove the lemma it suffices to give a bijection between $P$ and $Y$. Define the map $P \to Y$ by
\[
e_AH_Bf_C \mapsto e_A1_\lambda f_C,
\]
where $\lambda=(d-|A|-|B|-|C|)\varepsilon_1+B+\chi(e_Af_C)$. The inverse map is given by
\[
e_A1_\lambda f_C \mapsto e_AH_Bf_C,
\]
where $B=\lambda-\chi(e_Af_C)-\lambda_1 \varepsilon_1$. This completes the proof of the lemma.
\end{proof}

As $T$ surjects onto $S(m|n,d)$, it immediately follows from the previous two results that $Y$ is a basis for the Schur superalgebra and its integral form and that $T$ and $S$ are isomorphic.  Therefore we have proven Theorem~\ref{T:t1} and the following result. 
\begin{thm}\label{T:t1b}  The set
 \[
Y=\bigcup_{\lambda\in\Lmnd}\{e_A1_\lambda f_C\mid A,C \in \Pmn,  \chi(e_Af_C)\preceq \lambda \}
\] is a $\Q$-basis for $S(m|n,d)$ and a $\Z$-basis for $S(m|n,d)_{\Z}$.
\end{thm}

Finally, we note that there is another basis similar to $Y$ in which the $e$ and $f$ monomials are interchanged (see \cite[Theorem 2.3]{DG} where the analogous basis is denoted $Y_{-}$).

\subsection{A weight idempotent presentation} We also have an alternate presentation of the Schur superalgebra using weight idempotents.
\begin{thm}\label{T:t1c}  The Schur superalgebra $S(m|n,d)$ is generated by the homogeneous elements
\[
e_{1}, \dotsc , e_{m+n-1}, f_{1}, \dotsc , f_{m+n-1}, 1_{\lambda},
\] where $\lambda$ runs over the set $\Lmnd$ and the $\Z_{2}$-grading is given by setting  $\bar{e}_{m}=\bar{f}_{m}=\1$, $\overline{e}_{i}= \overline{f}_{i}=\0$ for $i \neq m$, and $\overline{1}_{\lambda}=\0$ for all $\lambda \in \Lmnd$.

The following are a complete set of relations:
 \begin{enumerate}
 \item [(R1$'$)]$1_{\lambda}1_{\mu}= \delta_{\lambda,\mu}1_{\lambda}$, $\sum_{\lambda \in \Lmnd} 1_{\lambda}=1$\\
\item [(R2$'$)]$e_{i}1_{\lambda}=  \begin{cases} 1_{\lambda + \alpha_{i}}e_{i}, &\text{if $\lambda + \alpha_{i} \in \Lmnd$}; \\
0, &\text{otherwise}.
\end{cases}$\\
\item [(R2$''$)]$f_{i}1_{\lambda}=  \begin{cases} 1_{\lambda - \alpha_{i}}f_{i}, &\text{if $\lambda - \alpha_{i} \in \Lmnd$}; \\
0, &\text{otherwise}.
\end{cases}$\\
\item [(R2$'''$)]$1_{\lambda}e_{i}=  \begin{cases} e_{i}1_{\lambda - \alpha_{i}}, &\text{if $\lambda - \alpha_{i} \in \Lmnd$}; \\
0, &\text{otherwise}.
\end{cases}$\\
\item [(R2$''''$)]$1_{\lambda}f_{i}=  \begin{cases} f_{i}1_{\lambda + \alpha_{i}}, &\text{if $\lambda + \alpha_{i} \in \Lmnd$}; \\
0, &\text{otherwise}.
\end{cases}$ \\
\item [(R3$'$)] $[e_{i},f_{j}]= \delta_{i,j} \sum_{\lambda \in \Lmnd}\left(\lambda_{j}-(-1)^{\overline{e}_{i} \cdot \overline{f}_{j}}\lambda_{j+1} \right)1_{\lambda}$.
 \end{enumerate} And relations $(R4)$ and $(R5)$ given in Theorem~\ref{T:t1}.
\end{thm}  The proof of Theorem~\ref{T:t1c} is identical to the analogous \cite[Theorem~2.4]{DG} so we omit it.

\section{Quantum Case}
The ground field is now the field of rational functions in the indeterminate $q$, $\Q(q)$. In this section all vector spaces will be defined over $\Q (q)$. 

\subsection{The quantum supergroup for $\mathfrak{gl}(m|n)$}  We have analogous results in the quantum setting.   The enveloping superalgebra $U$ is replaced by the quantized enveloping superalgebra $\Uq=U_q(\mathfrak{gl}(m|n))$ defined in \cite{W,Z}\footnote{Note that there are errors in \cite{Z} which are corrected in \cite{W}.}.  By definition $\U$ is given by generators and relations as follows.  The generators are:
\[
E_{1}, \dotsc , E_{m+n-1}, F_{1}, \dotsc , F_{m+n-1}, K_{1}^{\pm 1}, \dotsc , K_{m+n}^{\pm  1}.
\]  The $\Z_{2}$-grading on $\Uq$ is given by setting  $\overline{E}_{m}=\overline{F}_{m}=\1$, $\overline{E}_{a}= \overline{F}_{a}=\0$ for $a \neq m$, and $\overline{K}_{a}^{\pm 1}=\0$.  These generators are subject to relations $(Q1)-(Q5)$ in Theorem~\ref{T:t2}.

\subsection{The q-Schur superalgebra}\label{SS:qschur}  To define the $q$-Schur superalgebra, $\Sqmnd$, we need to introduce the analogue of the natural representation for $\Uq$.  Set $\Vq$ to be the $(m+n)$-dimensional vector space with fixed basis $v_{1}, \dotsc , v_{m+n}$. A $\Z_{2}$-grading on $\Vq$ is given by setting $\overline{v}_{a}=\overline{a}$, where we use the notation introduced in \eqref{E:parity}.  Before proceeding we set a convenient notation.  For $a=1, \dotsc , m+n$ we define
\begin{equation}\label{E:qadef}
q_a=q^{(-1)^{\overline{a}}}.
\end{equation}

The analogue of the natural representation, $\rho: \Uq  \to \End_{\Q (q)} (\Vq )$, is defined by
\begin{align}\label{E:Vqdef}
\rho(K_a)v_b&=q^{(\varepsilon_{a},\varepsilon_{b})}v_b=q_a^{\delta_{a,b}}v_b, \notag  \\
\rho(E_{a})v_{b}&=\delta_{a+1,b}v_a, \\
\rho(F_{a})v_b&=\delta_{a,b}v_{a+1}. \notag
\end{align}  The bilinear form used above is as in \eqref{E:bilinearform}.  It is a direct calculation to verify that this defines a representation of $\Uq$.

We define a comultiplication on $\Uq$ given on generators by
\begin{align}\label{E:quantumcomultiplication}
\Delta(E_{a}) &= E_{a} \otimes K_{a}^{-1}K_{a+1} + 1 \otimes E_{a}, \notag\\
\Delta(F_{a}) &= F_{a} \otimes 1  + K_{a}K_{a+1}^{-1}  \otimes F_{a}, \\
\Delta(K_{a}) &= K_{a} \otimes K_{a}.\notag
\end{align}
Using this comultiplication and the sign convention discussed in Section~\ref{SS:schursuperalgebra} we then have an action of $\Uq$ for any $d \geq 1$ on the $d$-fold tensor product of the natural module,
\[
\Vq^{\otimes d} := \Vq  \otimes \Vq  \otimes \dotsb \otimes \Vq .
\] That is, we obtain a superalgebra homomorphism
\begin{equation}\label{E:quantumrho}
\rho_d:\Uq\rightarrow \End_{\Q (q)}\left( \Vq^{\otimes d}\right).
\end{equation}

We define the \emph{$q$-Schur superalgebra} $\Sqmnd$ to be the image of $\rho_{d}$.  In particular, we can and will view it as a quotient of the superalgebra $\Uq$ and so a set of generators of $\Uq$ gives a set of generators for $\Sqmnd$ which are subject to possibly additional relations.

\subsection{A presentation of the $q$-Schur superalgebra}\label{SS:quantumrootvectors} We first introduce the quantum analogue of root vectors so as to more easily state the relations for the $q$-Schur superalgebra.
For $1 \leq a \neq b \leq  m+n$ we define the root vector $E_{a,b}$ recursively as follows.  For $a=1, \dotsc , m+n-1$ we set
\[
E_{a,a+1}:=E_{a} \text{ and } E_{a+1,a}:=F_{a}.
\]   If $|a-b| >1$, then $E_{a,b}$ is defined by setting
\begin{align}\label{E:qrootvectordef}
E_{a,b}= \begin{cases}E_{a,c}E_{c,b}-q_cE_{c,b}E_{a,c}, & \text{ if $a>b$};\\
                     E_{a,c}E_{c,b}-q_c^{-1}E_{c,b}E_{a,c}, & \text{ if $a<b$}.
\end{cases}
\end{align}
where $c$ can be taken to be an arbitrary index strictly between $a$ and $b$.  It is straightforward to see that $E_{a,b}$ is independent of the choice of $c$.  It is also straightforward to see that $E_{a,b}$ is homogeneous and of degree $\overline{\varepsilon_{a}-\varepsilon_{b}}$.

We can now give a presentation for $\Sqmnd$.  Note that the bilinear form used in the following relations is defined in \eqref{E:bilinearform} and the notation $q_{a}$ is as defined in \eqref{E:qadef}.
\begin{thm}\label{T:t2} The $q$-Schur superalgebra $S_{q}(m|n,d)$ is generated by the homogeneous elements
\[
E_{1}, \dotsc , E_{m+n-1}, F_{1}, \dotsc , F_{m+n-1}, K_{1}^{\pm 1}, \dotsc , K_{m+n}^{\pm  1}.
\] The $\Z_{2}$-grading is given by setting $\overline{E}_{m}=\overline{F}_{m}=\1$, $\overline{E}_{a}= \overline{F}_{a}=\0$ for $a \neq m$, and $\overline{K}^{\pm 1}_{a}=\0$.  These elements are subject to the following relations:

\begin{enumerate}
\item[(Q1)] For $M,\,N\in\{\pm1\}$ and $1\leq a,\, b\leq m+n$,
\[K_a^MK_b^N=K_b^NK_a^M,\]
and 
\[K_aK_a^{-1}=K_a^{-1}K_a=1;\]
\item[(Q2)]For $1\leq a\leq m+n$ and $1\leq b\leq m+n-1$
\begin{align*}
K_aE_{b,b+ 1}&=q^{(\varepsilon_{a}, \alpha_{b})}E_{b,b+1}K_a=q_a^{(\delta_{a,b}-\delta_{a,b+ 1})}E_{b,b+1}K_a, \\
K_aE_{b+1,b}&=q^{(\varepsilon_{a}, -\alpha_{b})}E_{b+1,b}K_a=q_a^{(\delta_{a,b+ 1}-\delta_{a,b})}E_{b+1,b}K_a;
\end{align*}
\item[(Q3)]For $1\leq a,\, b\leq m+n-1$
\[
[E_{a,a+1},E_{b+1,b}]=\delta_{a,b}\frac{K_aK_{a+1}^{-1}-K_a^{-1}K_{a+1}}{q_a-q_a^{-1}},
\]
and for $|a-b|>1$, we have the commutations
\[E_{a+1,a}E_{b+1,b}=E_{b+1,b}E_{a+1,a}\quad\text{and}\quad E_{a,a+1}E_{b,b+1}=E_{b,b+1}E_{a,a+1};\]
\item[(Q4)]
$E_{m,m+1}^2=E_{m+1,m}^2=0$; \\
\item[(Q5)]If neither $m$ nor $n$ is $1$, we have the following $U_q(\mathfrak{gl}(m|n))$ \emph{Serre relations}.

 For $a\neq m$, we have
    \begin{enumerate}
    \item[(a)] $E_{a+1,a}E_{a+2,a}=q_aE_{a+2,a}E_{a+1,a},$ \quad$1\leq a\leq m+n-2$,
    \item[(b)] $E_{a,a+1}E_{a,a+2}=q_aE_{a,a+2}E_{a,a+1},$ \quad$1\leq a\leq m+n-2$,
    \item[(c)] $E_{a+1,a-1}E_{a+1,a}=q_aE_{a+1,a}E_{a+1,a-1},$ \quad$2\leq a\leq m+n$,
    \item[(d)] $E_{a-1,a+1}E_{a,a+1}=q_aE_{a,a+1}E_{a-1,a+1},$ \quad$2\leq a\leq m+n$;
    \end{enumerate}

 For $a=m$, we have
    \[[E_{m+1,m},E_{m+2,m-1}]=[E_{m,m+1},E_{m-1,m+2}]=0.\]
If either $m=1$ or $n=1$, then these relations are omitted;\\
\item[(Q6)] $K_1K_2\cdots K_mK_{m+1}^{-1}K_{m+2}^{-1}\cdots K_{m+n}^{-1}=q^d$; \\
\item[(Q7)] $(K_a-1)(K_a-q_a)(K_a-q_a^2)\cdots (K_a-q_a^d)=0$, for all $1\leq a\leq m+n$. \\
\end{enumerate}
\end{thm}

\subsection{Strategy and simplifications}  As in the nonquantum case, the approach of \cite{DG} applies in our setting once the correct definitions and calculations are established.  Namely, let $\Tq$ be the algebra defined by the generators and relations of Theorem~\ref{T:t2}.  The basic line of argument is the same as before: we prove that relations $(Q1)$ through $(Q7)$ hold in $\Sq=\Sqmnd$ and so we have a surjective map $\Tq \to \Sq$ induced by the map $\rho_{d}$ given in \eqref{E:quantumrho}.  We then show this map is an isomorphism by showing via a series of calculations that the dimension of $\Tq$ is no more than the dimension of $\Sq$.  As it is no more difficult, we actually prove a slightly stronger result by working with a $\Z[q,q^{-1}]$-form.

As before we lighten the reading by using the same notation for elements of $\Uq$ and their images in the quotients $\Tq$ and $\Sq$.  We will make it clear in which algebra we are working whenever it is important to do so.  Furthermore, we can again make use of the fact that the quantum group associated to $\fg_{\0}$ is a subalgebra of $\Uq$ (as the subalgebra generated by $E_{a}, F_{a}$ ($a \neq m$) and $K_{1}^{\pm 1}, \dotsc , K_{m+n}^{\pm 1}$) and so calculations on purely even elements follow from the analogous results in the non-super setting.
\subsection{The new relations}  We first prove that relations $(Q6)$ and $(Q7)$ hold in $\Sq=\Sqmnd$ and, hence, the surjection $\rho_{d}: \Uq\rightarrow \Sq$ factors through $\Tq$.
\begin{lem}\label{L:Q6Q7relations}
Under the representation $\rho_d:\Uq\rightarrow \End(\Vq^{\otimes d})$, the images of the $K_a$ satisfy the relations $(Q6)$ and $(Q7)$. Moreover, the relation $(Q7)$ is the minimal polynomial of the image of $K_a$ in $\End(\Vq^{\otimes d})$.
\end{lem}

\begin{proof} Using the action of $\Uq$ on $\Vq$ given in \eqref{E:Vqdef} and on $\Vq^{\otimes d}$ via the comultiplication \eqref{E:quantumcomultiplication} and the sign convention discussed in Section~\ref{SS:schursuperalgebra}, the argument is as in the nonquantum case except that the calculations are done multiplicatively.  We point out that there is one subtlety (and it is the reason why our relations differ slightly from the analogous ones from \cite[Lemma 8.1]{DG}).  Namely, the action of $K_{a}$ when $a >m$ is the inverse of what might be expected.
\end{proof} 

\subsection{Divided powers and weight idempotents}  Let $\Aq$ denote $\Uq$, $\Tq$, or $\Sq$.  We now define various elements of $\Aq$ which are analogous to those defined in the nonquantum setting.

We first introduce notation for the quantum integers.  Given $n \in \Z_{\geq 0}$, let
\[
[n] =\frac{q^{n}-q^{-n}}{q-q^{-1}}
\] and
\[
[n]! = [n] \cdot [n-1] \cdot \dotsb \cdot [2] \cdot [1].
\]  It is helpful for calculations to note that $[n]$ is unchanged by the substitution $q \mapsto q^{-1}$ and, in particular, under the substitution $q \mapsto q_{a}$.

Given $x \in \Aq$ and $k \in \Z_{\geq 0}$, we define the \emph{$k$th divided power of $x$} by
\[
x^{(k)} = \frac{x^{k}}{[k]!}.
\]  In particular, the root vectors introduced in Section~\ref{SS:quantumrootvectors} have divided powers, $E_{a,b}^{(r)}$, for all $1 \leq a \neq b \leq m+n$ and $r \geq 0$.

If $1 \leq a,b \leq m+n$, then we set
\[
K_{a,b}= K_{a}K^{-1}_{b}.
\] For $t \in \Z_{\geq 0}$ and $c \in \Z$, we use the $q_{a}$ notation given in \eqref{E:qadef} and set
\[
\begin{bmatrix}
K_a;c  \\
t
\end{bmatrix}=\prod_{s=1}^t\frac{K_aq_a^{c-s+1}-K_a^{-1}q_a^{-c+s-1}}{q_a^s-q_a^{-s}} \text{ and }
\begin{bmatrix}
K_{a,b};c  \\
t
\end{bmatrix}=\prod_{s=1}^t\frac{K_{a,b}q_a^{c-s+1}-K_{a,b}^{-1}q_a^{-c+s-1}}{q_a^s-q_a^{-s}}.
\] For short, we write
\[
\begin{bmatrix}
K_a \\
t
\end{bmatrix} =
\begin{bmatrix}
K_a;0  \\
t
\end{bmatrix}.
\]

For $\lambda=(\lambda_a)\in \Lmn$, we write
\[
K_\lambda=\prod_{a=1}^{m+n} \begin{bmatrix}
K_a \\
\lambda_a
\end{bmatrix}.
\]
As the $K_{a}$ commute, the product can be taken in any order.  For $\lambda \in \Lmnd$ we introduce the shorthand
\[
1_\lambda:=K_\lambda
\] and because of part (b) of Proposition~\ref{P:p21} we call these \emph{weight idempotents}.

We define $\Aq^0$ as the subalgebra of $\Aq$ generated by
\[
K_a^{\pm1} \text{ and }\begin{bmatrix}
K_a  \\
t
\end{bmatrix}
\] for all $a=1, \dotsc , m+n$, $t \in \Z_{\geq 0} $.  We define $\Ares^{0}$ to be the $\AAA= \Z[q,q^{-1}]$-subalgebra of $\Aq^{0}$ generated by
\[
K_a^{\pm1} \text{ and }\begin{bmatrix}
K_a  \\
t
\end{bmatrix}
\] for all $a=1, \dotsc , m+n$, $t \geq 0$.  If $\Aq$ equals $\Tq$ or $\Sq$, then it is clear that $\Aq^0$ and $\Ares^{0}$ is the image of $\Uq^0$ and $\Ures^{0}$, respectively, under the quotient map.

  Now we investigate the structure of $\T^0$ and $\Tres^{0}$.  In the following proposition we continue our use of the notation $q_{a}$ introduced in \eqref{E:qadef}.

\begin{prop}\label{P:p21}
Define $\Iq^0$ to be the ideal of $\Uq^{0}$ generated by
\[
K_1K_2\cdots K_mK_{m+1}^{-1}\cdots K_{m+n}^{-1}-q^d
\]
and
\[
(K_a-1)(K_a-q_a)\cdots (K_a-q_a^d)
\]
for $a=1, \dotsc , m+n$. Then the following statements hold.
\begin{itemize}
\item[(a)] We have a superalgebra isomorphism $\Uq^{0}/\Iq^{0}\cong \T^0$.
\item[(b)] The set $\{1_\lambda\mid \lambda\in \Lmnd\}$ is a $\Q(q)$-basis for $\T^0$ and a $\Z[q,q^{-1}]$-basis for $\Tres^{0}$.  Moreover, they give a set of pairwise orthogonal idempotents which sum to the identity.
\item[(c)] $K_{\mu}=0$ for any $\mu \in \Lmn$ such that $|\mu|>d$.
\end{itemize}
\end{prop}
\begin{proof}
As these elements are purely even, the proof of \cite[Proposition~8.2]{DG} applies if we keep in mind the slight difference in $K_{a}$ when $a > m$ and that we should replace each $v$ in their argument by $q_a$.
\end{proof}

To state the next result we need to introduce the Gaussian binomial coefficient.
For $z \in \Z$, and $t\in \Z_{\geq 0}$, define
\begin{equation}\label{E:gaussian}
\begin{bmatrix}
z \\
t
\end{bmatrix}=\prod_{s=1}^t \frac{q^{z-s+1}-q^{-z+s-1}}{q^s-q^{-s}}.
\end{equation}
In the equations which follow one might expect $q_{a}$ to appear in the binomial coefficients.   However, the binomial coefficient is invariant under the map $q \mapsto  q^{-1}$ so this dependency is avoided.

\begin{prop}\label{p22}
Let $1\leq a\leq m+n$, $t\in \Z_{\geq 0}$, $c\in \Z$, $\lambda\in \Lmnd$, and $\mu \in \Lmn$. We have the following identities in the superalgebra $T^0$:
\begin{enumerate}
\item [(a)]$K_a^{\pm1}1_\lambda=q_a^{\pm\lambda_a}1_\lambda, \quad \displaystyle{\begin{bmatrix}
K_a;c  \\
t
\end{bmatrix}1_\lambda=\begin{bmatrix}
\lambda_a+c  \\
t
\end{bmatrix}1_\lambda}.$
\item [(b)] $K_\mu 1_\lambda=\lambda_\mu 1_\lambda$, where $\lambda_\mu=\displaystyle{\prod _a \begin{bmatrix}
\lambda_a \\
\mu_a
\end{bmatrix}}.$
\item [(c)] $K_\mu =\displaystyle{\sum_{\lambda \in \Lmnd} \lambda_\mu 1_\lambda}$.
\end{enumerate}
\end{prop}

\begin{proof}
As the elements are purely even, the argument from the proof of \cite[Proposition~8.3]{DG} carries over if we replace $v$ by $q_a$.
\end{proof}

\subsection{Commutation relations between root vectors and weight idempotents}  Recall that in Section~\ref{SS:quantumrootvectors} we defined root vectors $E_{a,b } \in \Uq$ for every $1 \leq a \neq b \leq m+n$.  As is our convention, we also write $E_{a,b}$ for their image in $\Tq$ and $\Sq$.  We now compute the commutation relations between root vectors and weight idempotents.
\begin{prop}\label{p23}
For any $\lambda \in \Lmnd$, and $\alpha=\varepsilon_b-\varepsilon_c \in \Phi$, we have the commutation formulas:
 \[E_{b,c}1_\lambda=\begin{cases}
            1_{\lambda+\alpha}E_{b,c}& \text{if $\lambda+\alpha \in \Lmnd$}\\
            0& \text{otherwise}
            \end{cases}\]
and similarly
 \[1_\lambda E_{b,c}=\begin{cases}
            E_{b,c}1_{\lambda-\alpha}& \text{if $\lambda-\alpha \in \Lmnd$}\\
            0& \text{otherwise.}
            \end{cases}\]
\end{prop}

\begin{proof}
The following identities are derived by direct computation.
\begin{equation}
\begin{bmatrix}
K_a;0  \\
1
\end{bmatrix}\begin{bmatrix}
K_a;-1  \\
\lambda_a
\end{bmatrix}=\begin{bmatrix}
\lambda_a+1 \\
1
\end{bmatrix}\begin{bmatrix}
K_a;0  \\
\lambda_a+1
\end{bmatrix},\label{eq210}
\end{equation}
\begin{equation}
\begin{bmatrix}
K_a;1 \\
\lambda_a
\end{bmatrix}=q_a^{\lambda_a}\begin{bmatrix}
K_a  \\
\lambda_a
\end{bmatrix}+q_a^{\lambda_a-1}K_a^{-1}\begin{bmatrix}
K_a \\
\lambda_a-1
\end{bmatrix}.\label{eq211}
\end{equation}

From the defining relation $(Q2)$, we can see that $K_a$ and $E_{b,c}$ commute if $a\neq b$ and $a\neq c$. Moreover,
\[K_bE_{b,c}=q_bE_{b,c}K_b.\]
This implies
\begin{equation}
E_{b,c}\begin{bmatrix}
K_b \\
\lambda_b
\end{bmatrix}=\begin{bmatrix}
K_b;-1 \\
\lambda_b
\end{bmatrix}E_{b,c}.\label{eq212}
\end{equation}
We also have:
\[K_cE_{b,c}=q_c^{-1}E_{b,c}K_c,\]
which implies
\begin{equation}
E_{b,c}\begin{bmatrix}
K_c \\
\lambda_c
\end{bmatrix}=\begin{bmatrix}
K_c;1 \\
\lambda_c
\end{bmatrix}E_{b,c}.\label{eq213}
\end{equation}
Then, for $\lambda\in \Lmnd$, and $b\neq c$, we have
\[E_{b,c}1_\lambda=\begin{bmatrix}
K_b;-1 \\
\lambda_b
\end{bmatrix}\begin{bmatrix}
K_c;1 \\
\lambda_c
\end{bmatrix}\prod_{l\neq b,c}\begin{bmatrix}
K_l \\
\lambda_l
\end{bmatrix}E_{b,c}.\]
Multiply both sides of the preceding equality by $\begin{bmatrix}
K_b \\
\lambda_b
\end{bmatrix}$ and use \eqref{eq210} to simplify the right-hand side and \eqref{eq212}, \eqref{eq213} to simplify the left-hand side. The result is:
\[E_{b,c}\begin{bmatrix}
K_b;1 \\
\lambda_b
\end{bmatrix}1_\lambda=\begin{bmatrix}
\lambda_b+1 \\
1
\end{bmatrix}\begin{bmatrix}
K_b \\
\lambda_b+1
\end{bmatrix}\begin{bmatrix}
K_c;1 \\
\lambda_c
\end{bmatrix}\prod_{l\neq b,c}\begin{bmatrix}
K_l \\
\lambda_l
\end{bmatrix}E_{b,c}.\]
Assuming $\lambda_c\geq 1$ and using \eqref{eq211}, we get
\[E_{b,c}1_\lambda=\begin{bmatrix}
\lambda_b+1 \\
1
\end{bmatrix}\begin{bmatrix}
K_b \\
\lambda_b+1
\end{bmatrix}\left(q_c^{\lambda_c}\begin{bmatrix}
K_c  \\
\lambda_c
\end{bmatrix}+q_c^{\lambda_c-1}K_c^{-1}\begin{bmatrix}
K_c \\
\lambda_c-1
\end{bmatrix}\right)\prod_{l\neq b,c}\begin{bmatrix}
K_l \\
\lambda_l
\end{bmatrix}E_{b,c}.\]
Thus, when $\lambda_c\geq 1$ we can multiply through in the above expression and apply Proposition~\ref{P:p21}$(c)$ to see that the first summand must be zero. The
above equality simplifies to
\[E_{b,c}1_\lambda=q_c^{\lambda_c-1}K_c^{-1}1_{\lambda+\alpha}E_{b,c}.\]
Now, by Proposition \ref{p22}$(a)$, $K_c^{-1}$ acts on $1_{\lambda+\alpha}$ as $q_c^{-(\lambda_c-1)}$. Thus we obtain the equality in the first part of the proposition in
the case $\lambda_c\geq 1$.

If $\lambda_c=0$, then the right-hand-side is zero by Proposition~\ref{P:p21}$(c)$. This proves the first part of the proposition. The proof of the second part is similar.
\end{proof}

\subsection{Commutation formulas between divided powers of root vectors}  We will need to know how divided powers of root vectors commute with each other.  To obtain this we use the \emph{PBW-Commutator Lemma} presented in \cite{W}. We first consider the case when both root vectors correspond to positive roots\footnote{Note that there is a typographic error in \cite[20(b)]{W} and that we have chosen to write signs in an equivalent but more symmetric fashion.}.
\begin{equation}\label{E:DeWitPositive}
E_{a,b}E_{c,d}=\begin{cases}
            (-1)^{\overline{E}_{a,b}\overline{E}_{c,d}}E_{c,d}E_{a,b}& \text{($b<c$ or $c<a<b<d$)}\\
            (-1)^{\overline{E}_{a,b} \overline{E}_{c,d}}q_bE_{c,d}E_{a,b}& \text{($a<c<b=d$)}\\
           (-1)^{\overline{E}_{a,b} \overline{E}_{c,d}}q_aE_{c,d}E_{a,b}& \text{($a=c<b<d$)}\\
            E_{a,d}+q_c^{-1}E_{c,d}E_{a,b}& \text{($b=c$)}\\
            (-1)^{\overline{E}_{a,b}\overline{E}_{c,d}}E_{c,d}E_{a,b}+(q_b-q_b^{-1})E_{a,d}E_{c,b}& \text{($a<c<b<d$)}
            \end{cases}
\end{equation}

 Before stating the result, we first observe that we can make the following assumptions.  First, since the case when both root vectors have divided power one is handled by \eqref{E:DeWitPositive}, we may assume that at least one of the powers is greater than one.  Second, if $\varepsilon_{a}-\varepsilon_{b}$ is an odd root, then by \cite[Section IV]{W} we have $E_{a,b}^{2}=0$.  That is, just as in the nonquantum case we may assume the odd root vectors have divided power at most one. Therefore, in what follows if the power of a root vector is one, then it may be even or odd; but if the power is greater than one, then we are implicitly assuming the root vector is even.  In particular, the combination of these two assumptions means that in each formula below at least one root vector is even and, hence, our formulas do not involve extra signs due to the $\Z_{2}$-grading.

Under the above assumptions lengthy but elementary inductive arguments using \eqref{E:DeWitPositive} imply the following commutator formulas for the divided powers of root vectors associated to positive roots. In these relations and the ones that follow we use the $q_{a}$ notation introduced in \eqref{E:qadef} and the Gaussian binomials introduced in \eqref{E:gaussian}.  The relations given here are analogous to those obtained by Xi for the quantum groups of simple Lie algebras \cite{X}.

\begin{prop}\label{P:commute}  Let $E_{a,b}$ and $E_{c,d}$ be two root vectors with $a<b$ and $c<d$, and let $M,N \geq 1$ satisfying the assumptions given above.  We then have the following commutation formulas.
\begin{enumerate}
\item If $b<c$ or $c<a<b<d$, then
\[
E_{a,b}^{(M)}E_{c,d}^{(N)}= E_{c,d}^{(N)}E_{a,b}^{(M)} .
\]
\item If $a=c<b<d$ or $a<c<b=d$, then
\[
E_{a,b}^{(M)}E_{c,d}^{(N)}= q_b^{MN} E_{c,d}^{(N)}E_{a,b}^{(M)} .
\]
\item If $a<b=c<d$, then
\[
E_{a,b}^{(M)}E_{c,d}^{(N)}=\sum_{t=0}^{\min(M,N)}q_{b}^{-(N-t)(M-t)}E_{c,d}^{(N-t)}E_{a,d}^{(t)}E_{a,b}^{(M-t)}.
\]
\item If $a<c<b<d$, then
\[
E_{a,b}^{(M)}E_{c,d}^{(N)}=\sum_{t=0}^{\min(M,N)}q_{b}^{\frac{t(t-1)}{2}}(q_b-q_b^{-1})^t[t]!E_{c,b}^{(t)}E_{c,d}^{(N-t)}E_{a,b}^{(M-t)}E_{a,d}^{(t)}.
\]
\end{enumerate}
\end{prop}

We note that from these commutator formulas we can derive a second set by solving for $E_{c,d}^{(N)}E_{a,b}^{(M)}$ and then interchanging $(a,b)$ and $(c,d)$.  Taken together with the formulas given in the proposition these give a complete set of commutator formulas for divided powers of positive root vectors.  That this is a complete set of formulas can easily be seen by considering the various possibilities for the subscripts (cf.\ \cite[Section 9]{DG}).

There is a similar set of commutator formulas for divided powers of negative root vectors.  They can be derived directly using the analogous results from \cite{W}.  Alternatively, $\Uq$ admits an antiautomorphism given by $E_{a} \mapsto F_{a}$, $F_{a}\mapsto E_{a}$, and $K_{a}\mapsto K_{a}^{-1}$.  Applying this map to the commutator relations for positive root vectors yields the commutator relations among negative root vectors.

\subsection{More commutation formulas}\label{SS:morecommformulas}  Finally we give the commutation formulas between a positive and a negative root vector.
 Let us assume $a<b$ and $c<d$.  Then from \cite{W} we have the following:
\begin{equation}\label{eq215}
E_{a,b}E_{d,c}=\begin{cases}
            (-1)^{\bar{E}_{a,b}\bar{E}_{d,c}}E_{d,c}E_{a,b}&  \text{($b\leq c$ or $c<a<b<d$)}\\
            (-1)^{\bar{E}_{a,b}\bar{E}_{d,c}}E_{d,c}E_{a,b}+K_{c,b}E_{a,c}&  \text{($a<c<b=d$)}\\
           (-1)^{\bar{E}_{a,b}\bar{E}_{d,c}}E_{d,c}E_{a,b}- (-1)^{\bar{E}_{a,b}\bar{E}_{d,c}}K_{a,b}E_{d,b}& \text{($a=c<b<d$)}\\
           (-1)^{\bar{E}_{a,b}\bar{E}_{d,c}}E_{d,c}E_{a,b}+(q_a-q_a^{-1})^{-1}(K_{a,b}-K_{a,b}^{-1})&  \text{($a=c$ and $b=d$)}\\
            (-1)^{\bar{E}_{a,b}\bar{E}_{d,c}}E_{d,c}E_{a,b}-(q_b-q_b^{-1})K_{c,b}E_{a,c}E_{d,b}& \text{($a<c<b<d$)}
            \end{cases}
\end{equation}

Using these and elementary inductive arguments yields the following formulas.  Note that the assumptions on divided powers of root vectors stated before Proposition~\ref{P:commute} apply here as well.

\begin{prop} Let $E_{a,b}$ and $E_{d,c}$ be two root vectors  with $a<b$ and $c<d$,  and let $M,N \geq 1$. We then have the following commutation formulas.
\begin{enumerate}
\item  If $b\leq c$ or $c<a<b<d$, then
\[
E_{a,b}^{(M)}E_{d,c}^{(N)}= E_{d,c}^{(N)}E_{a,b}^{(M)}.
\]
\item If $a<c<b=d$
\[
E_{a,b}^{(M)}E_{d,c}^{(N)}= \sum_{t=0}^{\min(M,N)} q_{b}^{-t(N-t)}E_{d,c}^{(N-t)}K_{c,d}^{t}E_{a,b}^{(M-t)}E_{a,c}^{(t)}.
\]
\item If $a=c<b<d$, then
\[
E_{a,b}^{(M)}E_{d,c}^{(N)}= \sum_{t=0}^{\min(M,N)} (-1)^tq_{b}^{-t(M-1-t)}E_{d,b}^{(t)}E_{d,c}^{(N-t)}K_{a,b}^{t}E_{a,b}^{(M-t)}.
\]
\item  If $a<b$, then
\[
E_{a,b}^{(M)}E_{b,a}^{(N)}= \sum_{t=0}^{\min(M,N)} E_{b,a}^{(N-t)}\begin{bmatrix} K_{a,b}; 2t-M-N \\ t \end{bmatrix} E_{a,b}^{(M-t)}.
\]
\item If $a<c<b<d$, then
\[
E_{a,b}^{(M)}E_{d,c}^{(N)}= \sum_{t=0}^{\min(M,N)} (-1)^tq_{b}^{\frac{-t(2N-3t-1)}{2}}(q_b-q_b^{-1})^{t}[t]!E_{d,c}^{(N-t)}E_{d,b}^{(t)}K_{c,b}^{t}E_{a,b}^{(M-t)}E_{a,c}^{(t)}.
\]

\end{enumerate}
\end{prop}

We can use the antiautomorphism on $\Uq$ defined in the previous section along with simple calculations to derive additional identities (cf.\ \cite[Section 9]{DG}).  In this way we obtain a complete set of commutation relations involving a positive root vector to the left of a negative root vector.
There are similar commutation formulas for the case of a negative root vector followed by a positive root vector. These can be obtained from the above formulas by solving
for the term $E_{d,c}^{(N)}E_{a,b}^{(M)}$. The new formulas will be of a similar form.

Taking all possible formulas we obtain the commutation formulas for divided powers of root vectors. The interested reader can derive the complete set.
\subsection{An $\AAA$-form for $\Uq$}\label{SS:LusztigAform}  Recall that Lusztig defined an $\AAA=\Z[q,q^{-1}]$-form for $U_{q}(\fg )$ whenever $\fg$ is a semisimple Lie algebra.  We define an analogous $\AAA$-form for $\Uq$.  Let $\Ures$ denote the $\AAA$-subsuperalgebra of $\Uq$ generated by
\begin{equation*}
\left\{E_{a,b}^{(M)}, K_{a}^{\pm 1}, \begin{bmatrix} K_{a} \\ t \end{bmatrix} \mid 1 \leq a \neq b \leq m+n, M,t \in \Z_{\geq 0}\right\}.
\end{equation*}

 Fix an order on the root system $\Phi^+$ and let $\Pmn$ be as in \eqref{E:Pmndef}. For $A=(A(\alpha)) \in \Pmn$, we define
\begin{align*}
E_A&=\prod_{\alpha=\varepsilon_a-\varepsilon_b \in \Phi^+}E_{a,b}^{\left( A(\alpha)\right)},\\
 F_A&=\prod_{\alpha=\varepsilon_a-\varepsilon_b \in
 \Phi^+}E_{b,a}^{\left( A(\alpha)\right)},
\end{align*} where the product is taken according to the fixed order on $\Phi^{+}$.

There is a known basis for the analogously defined $\AAA$-form for $\Uq_{q}(\fg_{\0})$ following from Lusztig's basis for $\Uq_{q}(\mathfrak{sl}(n))$ \cite[Theorem 4.5]{Lusztig} (see also \cite{X}).   Using this basis and the quantum commutator formulas given in the previous section it follows that $\Ures$ has an $\AAA$-basis given by the set
\begin{equation}\label{E:quantumPBW}
\left\{ E_{A} \prod_{a=1}^{m+n}\left( K_{a}^{\sigma_{a}} \begin{bmatrix}  K_{a} \\ \mu_{a} \end{bmatrix}\right)F_{C} \mid A, C \in \Pmn, \sigma_{1}, \dotsc , \sigma_{m+n} \in \{0,1 \}, \mu  \in \Lmn \right\}.
\end{equation}  In particular this gives a basis for $\Uq$ after extending scalars (compare with \cite[Proposition~1]{Z}).

If $\Aq$ is $\Sq$ or $\Tq$, then we define $\Ares$ to be the image of $\Ures$ under the quotient map.  In particular $\Ares$ is a $\Z[q,q^{-1}]$-subsuperalgebra of $\Aq$ and (the image under the quotient map of) the set given in \eqref{E:quantumPBW} spans $\Ares$. For short we call $\Sqmnd_{\AAA}$ the \emph{integral $q$-Schur superalgebra}.

\subsection{Quantum Kostant monomials and content functions}  We now define the quantum analogue of the Kostant monomials.  Any finite product of nonzero elements of the form
 \[
E_{a,b}^{(M)}, \quad K_a^{\pm1}, \quad \begin{bmatrix}
K_a  \\
t
\end{bmatrix},
\] where $1\leq a\neq b\leq m+n$ and $M,t\in\Z_{\geq 0}$, will be called a \emph{Kostant monomial}.

We also define content functions as before.
  Namely, the content function
\begin{equation}\label{E:qcontentdef}
\chi: \left\{ \text{Kostant monomials}\right\} \rightarrow \bigoplus_{i=1}^{m+n}\Z\varepsilon_i
\end{equation}
is given on generators by declaring for
  $\alpha=\varepsilon_a-\varepsilon_b \in \Phi$, $M, N \in \N$, and $t\in \Z_{\geq0}$ that

\begin{align*}
\chi\left( E_{a,b}^{(M)}\right)&=M\varepsilon_{\max(a,b)}, \\
 \chi\left( K_a\right) &= \chi\left( K_a^{-1}\right)=\chi\left( \begin{bmatrix}
K_a\\
t
\end{bmatrix}\right)=0.
\end{align*}
For general monomials we again use the formula $\chi(XY)=\chi(X)+\chi(Y)$ whenever $X, Y$ are Kostant monomials.

We also define the left content, $\chi_L$, and right content, $\chi_R$, by declaring on generators that
\begin{align*}
\chi_L(E_{a,b}^{(M)})&=M\varepsilon_a, \\
\chi_L(K_a)&=\chi_L(K_a^{-1})=\chi_L\left( \begin{bmatrix}
K_a  \\
t
\end{bmatrix}\right) = 0 \\
\chi_R(E_{a,b}^{(M)})&=M\varepsilon_b, \\
\chi_R(K_a)&=\chi_R(K_a^{-1})=\chi_R\left( \begin{bmatrix}
K_a  \\
t
\end{bmatrix}\right)=0,
\end{align*}
and again using the rule $\chi_L(XY)=\chi_L(X)+\chi_R(Y)$ (similarly for $\chi_R$) whenever $X$ and $Y$ are Kostant monomials.  We again use \eqref{E:contenttoweights} to view outputs of the content functions as elements of $\Lmn$.

\subsection{A basis for the $q$-Schur superalgebra}\label{SS:integralquantum}  We can now state the quantum analogue of Theorem~\ref{T:t1b}.

\begin{thm}\label{T:integralquantum}  The integral $q$-Schur superalgebra is the $\AAA$-subalgebra of $\Sqmnd$ generated by
\[
E_{a}^{(M)}, F_{a}^{(M)}, \begin{bmatrix} K_{b} \\ t,
\end{bmatrix}
\] where $1 \leq a \leq m+n-1$, $1 \leq b \leq m+n$, and $M \in  \Z_{\geq 0}$.  Moreover, the set
\[
\Yq=\bigcup_{\lambda\in\Lmnd}\left\{ E_A1_\lambda F_C\mid A,C \in \Pmn, \chi(E_AF_C)\preceq \lambda\right\}
\]  forms a $\Q (q)$-basis of $\Sqmnd$ and an $\AAA$-basis of $\Sqmnd_{\AAA}$.
\end{thm}

We remark that, as in Section~\ref{SS:integralformbasis}, the set $\Yq$ has alternate descriptions using the left and right content functions.  Applying the antiautomorphism of $\Uq$ yields a similar basis in which the positions of the $E$ and $F$ terms are swapped; that is, the analogue of $\Yq_{-}$ in \cite{DG}.

 \begin{prop}\label{p24}
 The set $\Yq$ spans the superalgebra $\T$.
 \end{prop}

  \begin{proof} The proof is exactly analogous to the proof of Proposition~\ref{p4} and the proof of \cite[Proposition 9.1]{DG}. One again argues by induction on degree and content using the above commutation formulas to write an arbitrary Kostant monomial as a $\Z[q,q^{-1}]$-linear combination of elements of $\Yq$.  The coefficients in our commutation formulas are slightly different, but they are still elements of $\Z[q,q^{-1}]$ and so this does not affect the substance of the argument.
 \end{proof}

  \begin{lem}\label{l24}
 The cardinality of the set $\Yq$ is equal to the dimension of $\s=\Sqmnd$.
 \end{lem}
 \begin{proof} It is known that the dimension of $\Sqmnd$ over $\Q(q)$ equals the dimension of $S(m|n,d)$ over $\Q$.  This is established, for example, in the proof of \cite[Proposition 4.3]{M}.  This can also be seen as an outcome of \cite[Theorem 9.7]{DR}. The result then follows by the proof of Lemma~\ref{l4}.
\end{proof}

Theorems~\ref{T:t2}~and~\ref{T:integralquantum} now follow as in the nonquantum case.

\subsection{A weight idempotent presentation}  We also have a quantum analogue of Theorem~\ref{T:t1c} which gives the $q$-Schur superalgebra by generators and relations using the weight idempotents.

\begin{thm}\label{T:quantumidempotent}  The $q$-Schur superalgebra $S_{q}(m|n,d)$ is generated by the homogeneous elements
\[
E_{1}, \dotsc , E_{m+n-1}, F_{1}, \dotsc , F_{m+n-1}, 1_{\lambda},
\] where $\lambda$ runs over the set $\Lmnd$.  The $\Z_{2}$-grading is given by setting $\overline{E}_{m}=\overline{F}_{m}=\1$, $\overline{E}_{a}= \overline{F}_{a}=\0$ for $a \neq m$,  and $\overline{1}_{\lambda}=\0$ for all $\lambda \in \Lmnd$.

These generators are subject only to the relations:
 \begin{enumerate}
 \item [(Q1$'$)]$1_{\lambda}1_{\mu}= \delta_{\lambda,\mu}1_{\lambda}$, $\sum_{\lambda \in \Lmnd} 1_{\lambda}=1$\\
\item [(Q2$'$)]$E_{a}1_{\lambda}=  \begin{cases} 1_{\lambda + \alpha_{a}}E_{a}, &\text{if $\lambda + \alpha_{a} \in \Lmnd$}; \\
0, &\text{otherwise}.
\end{cases}$\\
\item [(Q2$''$)]$F_{a}1_{\lambda}=  \begin{cases} 1_{\lambda - \alpha_{a}}F_{a}, &\text{if $\lambda - \alpha_{a} \in \Lmnd$}; \\
0, &\text{otherwise}.
\end{cases}$\\
\item [(Q2$'''$)]$1_{\lambda}E_{a}=  \begin{cases} E_{a}1_{\lambda - \alpha_{a}}, &\text{if $\lambda - \alpha_{a} \in \Lmnd$}; \\
0, &\text{otherwise}.
\end{cases}$\\
\item [(Q2$''''$)]$1_{\lambda}F_{a}=  \begin{cases} F_{a}1_{\lambda + \alpha_{a}}, &\text{if $\lambda + \alpha_{a} \in \Lmnd$}; \\
0, &\text{otherwise}.
\end{cases}$
\item [(Q3$'$)] $[E_{a},F_{b}]= \delta_{a,b} \sum_{\lambda \in \Lmnd}\left(\lambda_{b}-(-1)^{\overline{E}_{a} \overline{F}_{b}}\lambda_{b+1} \right)1_{\lambda}$.
 \end{enumerate} And relations $(Q4)$ and $(Q5)$ given in Theorem~\ref{T:t2}.
\end{thm}

Theorem~\ref{T:quantumidempotent} is proven just as in the nonquantum case and as in the proof of \cite[Theorem 3.4]{DG}.

\section{The $q$-Schur Superalgebra as an Endomorphism Superalgebra}

\subsection{Quantum Schur-Weyl duality}\label{SS:quantumschurweyl}

There is a natural signed action of the Iwahori-Hecke algebra associated to the symmetric group on $d$ letters, $\Hq = \Hq (\Sigma_{d})$, on $\V^{\otimes d}$.   In \cite{M} Mitsuhashi defines the $q$-Schur superalgebra as the superalgebra
\[
\widetilde{\Sq}:=\widetilde{\Sq}(m|n,d) = \End_{\Hq}(\Vq^{\otimes d}).
\]    The main result of \cite{M} is to establish a Schur-Weyl duality between this endomorphism algebra and the Iwahori-Hecke algebra.  However, it is not immediately obvious the $q$-Schur superalgebra defined in this paper as a quotient of $\Uq$ coincides with the one used there.    We now reconcile this difference.

Recall that we have a fixed homogeneous basis $v_{1}, \dotsc , v_{m+n}$ for $\Vq$ and this defines a homogeneous basis $\{v_{i_{1}}\otimes \dotsb \otimes v_{i_{d}} \mid  1 \leq i_{1}, \dotsc , i_{d} \leq m+n \}$ for $\Vq^{\otimes d}$.  Define a map $\sigma_{d}: \Vq^{\otimes d}\to \Vq^{\otimes d}$ by
\[
\sigma_{d} (v_{i_{1}} \otimes \dotsb \otimes v_{i_{d}}) = (-1)^{\bar{v}_{i_{1}}+\dotsb +\bar{v}_{i_{d}}} v_{i_{1}}\otimes \dotsb \otimes v_{i_{d}}.
\]   It is easily seen that $\sigma_{d}$ commutes with the action of $\Hq$ on $\Vq^{\otimes d}$ defined in \cite{M}.

Let $\Uq^{\sigma}$ denote the quantum group associated to $\mathfrak{gl}(m|n)$ in \cite{BKK,M}.  This algebra is generated by elements $e_{1}, \dotsc ,e_{m+n-1}$, $f_{1}, \dotsc , f_{m+n-1}$, and $q^{h}$ (where $h$ ranges over the elements of the dual weight lattice), along with an element denoted by $\sigma$.  For each $d\geq 1$, let
\[
\tilde{\rho}_{d}: \Uq^{\sigma} \to \End_{\Q (q)}\left(\Vq^{\otimes d} \right)
\] denote the homomorphism given in \cite[Equation (3.2)]{M}.  Mitsuhashi proves in \cite[Theorem 4.4]{M} that $\widetilde{\Sq}=\tilde{\rho}_{d}\left( \Uq^{\sigma}\right).$  For short we write $\Sq$ for the $q$-Schur superalgebra defined in Section~\ref{SS:qschur} as a quotient of $\Uq$.  We claim that $\widetilde{\Sq}= \Sq$.  When $d =1$, it is straightforward to see that the action of the generators $e_{a}$, $f_{a}$, $q^{h}$, coincide with the action of our $E_{a}$, $F_{a}$, and $K_{a}^{\pm 1}$.  More generally, this remains true for $d\geq 1$ once we take into account the fact that the difference in the coproducts is exactly explained by the fact that we use the sign convention whereas Mitsuhashi does not but instead introduces the element $\sigma$ (which acts on $\Vq^{\otimes d}$ by $\sigma_{d}$).

Thus $\Sq \subseteq \widetilde{\Sq}$.  It only remains to account for the extra generator $\sigma$ in $\Uq^{\sigma}$.   That is, since $\sigma$ acts on $\Vq^{\otimes d}$ by the map $\sigma_{d}$, we need to show that $\sigma_{d}$ lies in $\Sq$.  The next lemma shows that it lies in the image of $\rho_{d}$ and, hence, in $\Sq$.

\begin{lem}\label{L:Sigma}  For each $d \geq 1$, there exists $x_{d} \in \Uq$ so that $\rho_{d}(x_{d}) = \sigma_{d}$.
\end{lem}

\begin{proof} It suffices to construct an element of $\Uq$ whose action on our basis for $\Vq^{\otimes d}$ coincides with the action of $\sigma_{d}$.  We build this element up in several steps.  First,  for $0\leq s \leq d$ and $1\leq a \leq m+n$ we use the notation given in \eqref{E:qadef} and \eqref{E:parity} to define $\omega_{s,a} \in \Uq$ by
\[
\omega_{s,a}=\frac{(K_a-1)(K_a-q_a)\cdots(K_a-q_a^{s-1})(K_a-q_a^s+(-1)^{s \cdot \overline{a}})(K_a-q_a^{s+1})\cdots(K_a-q_a^d)}{(q_a^s-1)(q_a^s-q_a)\cdots(q_a^s-q_a^{s-1})(q_a^s-q_a^{s+1})\cdots(q_a^s-q_a^d)}.
\]   Given $ 1 \leq a  \leq  m+n$ we define a function,
\[
r_{a}: \left\{ v_{i_{1}} \otimes \dotsb \otimes v_{i_{d}} \mid  1 \leq i_{1}, \dotsc , i_{d} \leq m+n  \right\} \to \left\{0,1,\dotsc ,d \right\},
\]
which counts the occurrences of $v_{a}$ in $v_{i_{1}}\otimes \dotsb \otimes v_{i_{d}}$.  That is, it is defined by
\[
r_{a} = r_{a}(v_{i_{1}}\otimes \dotsb \otimes v_{i_{d}}) = | \left\{t = 1, \dotsc , d \mid i_{t}=a \right\}|.
\]  Then a direct calculation (cf.\ the calculation used to prove relation $(Q7)$ in Lemma~\ref{L:Q6Q7relations}) shows that
\begin{align*}
\omega_{s,a}(v_{i_1}\otimes v_{i_2}\otimes \cdots \otimes v_{i_d})= \begin{cases}
                       (-1)^{r_{a} \cdot\overline{a}}(v_{i_1}\otimes v_{i_2}\otimes \cdots \otimes v_{i_d}), & \text{ if $s=r_a$};\\
                       0, & \text{ if $s\neq r_a$}.
\end{cases}
\end{align*}

Now, for $1 \leq a \leq m+n$ define $\Omega_{a}\in \Uq$ by
\[
\Omega_a=\sum_{s=0}^{d}\omega_{s,a}.
\]  It then follows that for any basis vector $v_{i_{1}}\otimes \dotsb \otimes v_{i_{d}}$ we have
\[
\Omega_a(v_{i_1}\otimes \cdots \otimes v_{i_d})=(-1)^{r_{a}\cdot \overline{a}}(v_{i_1}\otimes \cdots \otimes v_{i_d}).
\]

Finally we define $\Omega \in \Uq$ to be the element
\begin{equation*}
\Omega=\prod_{a=1}^{m+n}\Omega_a.
\end{equation*}
It follows that we have
\begin{align*}
\Omega(v_{i_1}\otimes \cdots \otimes v_{i_d}) & =\left( \prod_{a=1}^{m+n}(-1)^{r_{a}\cdot\overline{a}}\right)(v_{i_1}\otimes \cdots \otimes v_{i_d}) \\
           & =(-1)^{r_{1}\cdot\overline{1}+\dotsb +r_{m+n}\cdot \overline{m+n}}(v_{i_1}\otimes \cdots \otimes v_{i_d}) \\
         & =(-1)^{\overline{v}_{i_{1}}+\dotsb +\overline{v}_{m+n}}(v_{i_1}\otimes \cdots \otimes v_{i_d})
\end{align*}
for every basis element $v_{i_{1}} \otimes \dotsb \otimes v_{i_{d}}$.  That is, as desired, $\Omega \in \Uq$ acts as $\sigma_{d}$ on $\Vq^{\otimes d}$.
\end{proof}


\begin{thebibliography}{99}
\bibitem{BLM} A.~A.~Beĭlinson, G.~Lusztig, R.~MacPherson, A geometric setting for the quantum deformation of $\operatorname{GL}_{n}$, \emph{Duke Math.\  J.} \textbf{61} (1990), no. 2, 655--677.
\bibitem{BKK} G.~Benkart, S.-J.~Kang, M.~Kashiwara, Crystal bases for the quantum superalgebra $U_{q}(gl(m,n))$, \emph{J.\  Amer.\  Math.\  Soc.} \textbf{13} (2000), no. 2, 295--331.
17B37
\bibitem{BR} A.~Berele and A.~Regev, Hook Young diagrams with applications to combinatorics and to representations of Lie superalgebras, \emph{Advances in Math.}, \textbf{64} (1987), 118--175.
\bibitem{BK} J.~Brundan and J.~Kujawa, A new proof of the Mullineux conjecture, \emph{J.\  Algebraic Combin.} \textbf{18} no 1. (2003), 13--39.
\bibitem{CPS} E.~Cline, B.~Parshall, L.~Scott, Finite-dimensional algebras and highest weight categories, \emph{J.\  Reine Angew.\  Math.} \textbf{391} (1988), 85--99.
\bibitem{DDPW}  B.~Deng, J.~Du, B.~Parshall, J.~Wang, Finite dimensional algebras and quantum groups, Mathematical Surveys and Monographs \textbf{150}, American Mathematical Society, Providence, RI, 2008.
\bibitem{W} D.~De~Wit,  A Poincare-Birkhoff-Witt commutator lemma for $U_q[gl(m|n)]$, \emph{J.\  Math.\  Phys.} \textbf{44} (2003), no. 1, 315--327.
\bibitem{D} S.~Donkin, Symmetric and exterior powers, linear source of modules and representations of Schur superalgebras, \emph{Proc.\  London Math.\  Soc.} (3) \textbf{83}, (2001), no. 3, 647--680.
\bibitem{DG} S.~Doty and A.~Giaquinto, Presenting Schur algebras, \emph{Int.\  Math.\  Res.\  Not.} \textbf{2002}, no. 36, 1907--1944.
\bibitem{DR} J.~Du and H.~Rui, Quantum Schur superalgebras and Kazhdan-Lusztig combinatorics,  \emph{J.\  Pure Appl.\  Algebra} \textbf{215} (2011), no. 11, 2715--2737.
\bibitem{Green}  R.~M.~Green, q-Schur algebras as quotients of quantized enveloping algebras, \emph{J.\  Algebra} \textbf{185} (1996), 660--687.
\bibitem{K} V.~Kac, Lie superalgebras, \emph{Adv.\  Math.} \textbf{26} (1977), 8--96.
\bibitem{kujawa} J.~Kujawa, Crystal structures arising from representations of $\operatorname{GL}(m|n)$, \emph{Represent.\  Theory} \textbf{10} (2006), 49--85.
\bibitem{Lauda} A.~Lauda, An introduction to diagrammatic algebra and categorified quantum sl(2), arXiv:1106.2128, (2011).
\bibitem{Li}  Y.~Li,  Semicanonical bases for Schur algebras, \emph{J.\  Algebra} \textbf{324} (2010), no. 3, 347--369.
\bibitem{LS} D.~Leĭtes, V.~Serganova, Defining relations for classical Lie superalgebras, I: Superalgebras with Cartan matrix or Dynkin-type diagram. \emph{Topological and geometrical methods in field theory} (Turku, 1991), 194--201, World Sci. Publ., River Edge, NJ, 1992.
\bibitem{Lusztig} G.~Lusztig, Finite-dimensional Hopf algebras arising from quantized universal enveloping algebra, \emph{J.\ Amer.\ Math.\  Soc.} \textbf{3} (1990), no.\  1, 257--296.
17B37 (16W30 20G40)
\bibitem{MSV} M.~Mackaay, M.~Stosic, P.~Vaz, A diagrammatic categorification of the q-Schur algebra, arXiv:1008.1348, \emph{Quantum Topology}, to appear.
\bibitem{M} H.~Mitsuhashi, Schur-Weyl reciprocity between the quantum superalgebra and the Iwahori-Hecke algebra, \emph{Algebr.\  Represent.\  Theory} \textbf{9} (2006), no. 3, 309--322.
\bibitem{Moon} D.~Moon, Highest weight vectors of irreducible representations of the quantum superalgebra $U_{q}(gl(m,n))$, \emph{J.\  Korean Math.\  Soc.} \textbf{40} (2003), no. 1, 1--28.
\bibitem{Ol} G.~I.~Olshanski, Quantized universal enveloping superalgebra of type Q and a super-extension of the Hecke algebra, \emph{Lett.\  Math.\  Phys.} \textbf{24} (1992), no. 2, 93--102.
\bibitem{Sergeev} A.~Sergeev, Tensor algebra of the identity representation as a module over the Lie superalgebras Gl(n,m) and Q(n), \emph{Mat.\  Sb.\  (N.S.)} \textbf{123} (165) (1984), no. 3, 422--430.
\bibitem{X} N.~Xi, A commutation formula for root vectors in quantized enveloping algebras, \emph{Pacific J.\  Math.} \textbf{189} (1999), no. 1, 179--199.
\bibitem{Z2} R.~B.~Zhang, Serre presentations of Lie superalgebras, arXiv:1101.3114, (2011).
\bibitem{Z} \bysame, Finite dimensional irreducible representations of the quantum supergroup $U_q(gl(m/n))$, \emph{J.\  Math.\  Phys.}  \textbf{34} (1993), no. 3, 1236--1254.
\end{thebibliography}
\end{document}